\documentclass[a4paper]{article}
\usepackage{makecell}
\usepackage{bbm}

\makeatletter
\let\c@figure\c@table
\let\ftype@figure\ftype@table

\usepackage{enumitem}
\usepackage{calc}

\usepackage{lmodern} %Type1-font for non-english texts and characters
\usepackage{graphicx} %%For loading graphic files

\usepackage{amsmath, accents}
\usepackage{amssymb,tikz}
\usepackage{pict2e}
\usepackage{amsthm}
\usepackage{amscd}
\usepackage{mathrsfs}
\usepackage{todonotes}
\usetikzlibrary{calc} % This is to add two points in tikzpictures

\usepackage{yfonts}

\usepackage{marvosym}
\usepackage[percent]{overpic}

\newtheorem{theorem}{Theorem} [section]

\newtheorem{lemma}[theorem]{Lemma}

\newtheorem{assumptions}[theorem]{Assumptions}

\theoremstyle{definition}

\newtheorem{remark}[theorem]{Remark}

\DeclareMathOperator{\tr}{Tr}

\newcommand{\C}{\mathbb{C}}

\newcommand{\R}{\mathbb{R}}
\newcommand{\N}{\mathbb{N}}
\newcommand{\PP}{\mathbb{P}}
\newcommand{\E}{\mathbb{E}}

\newcommand{\im}{\text{\upshape Im\,}}

\let\oldbibliography\thebibliography
\renewcommand{\thebibliography}[1]{\oldbibliography{#1}
\setlength{\itemsep}{-0.5pt}}

\def\XXint#1#2#3{{\setbox0=\hbox{$#1{#2#3}{\int}$}
\vcenter{\hbox{$#2#3$}}\kern-.5\wd0}}

\usepackage{tikz}
\usetikzlibrary{arrows}
\usetikzlibrary{decorations.pathmorphing}
\usetikzlibrary{decorations.markings}
\usetikzlibrary{patterns}
\usetikzlibrary{automata}
\usetikzlibrary{positioning}
\usepackage{tikz-cd}
\tikzset{->-/.style={decoration={
				markings,
				mark=at position #1 with {\arrow{latex}}},postaction={decorate}}}
	
	\tikzset{-<-/.style={decoration={
				markings,
				mark=at position #1 with {\arrowreversed{latex}}},postaction={decorate}}}
				
\tikzset{
	master/.style={
		execute at end picture={
			\coordinate (lower right) at (current bounding box.south east);
			\coordinate (upper left) at (current bounding box.north west);
		}
	},
	slave/.style={
		execute at end picture={
			\pgfresetboundingbox
			\path (upper left) rectangle (lower right);
		}
	}
}

\usetikzlibrary{shapes.misc}\tikzset{cross/.style={cross out, draw, 
         minimum size=2*(#1-\pgflinewidth), 
         inner sep=0pt, outer sep=0pt}}

\allowdisplaybreaks	

\numberwithin{equation}{section}

\def\bigO{{\cal O}}

\makeatletter
\newcommand{\oset}[3][0ex]{%
  \mathrel{\mathop{#3}\limits^{
    \vbox to#1{\kern-2\ex@
    \hbox{$\scriptstyle#2$}\vss}}}}
\makeatother

\usepackage{tocloft}
\usepackage[colorlinks=true]{hyperref}
\hypersetup{urlcolor=blue, citecolor=red, linkcolor=blue}

\begin{document}
\title{Largest gaps between bulk eigenvalues of unitary-invariant \\ random Hermitian matrices}
\author{Christophe Charlier}

\maketitle

\begin{abstract}
We study $n\times n$ random Hermitian matrix ensembles that are invariant under unitary conjugation. Let $I$ be a finite union of intervals lying in the bulk, and let $m_{k}^{(n)}$ be the $k$-th largest gap between consecutive eigenvalues lying in $I$. We prove that the rescaled gap $\smash{\tau_{k}^{(n)}}$, which is defined by
\begin{align*}
m_{k}^{(n)} = \frac{1}{2\pi \inf_{I}\rho} \bigg( \frac{\sqrt{32 \log n}}{n} + \frac{3q-8}{2q} \frac{  \log(2\log n)}{n \sqrt{2\log n}} + \frac{4\tau_{k}^{(n)}}{n \sqrt{2\log n}} \bigg),
\end{align*}
converges in distribution as $n\to +\infty$ to a gamma-Gumbel random variable that is shifted by an explicit constant $c_{V,I}$ depending only on $I$ and on the potential $V$. Here $\rho$ is the density of the equilibrium measure and $q\in \N_{>0}$ is the highest order at which $\rho(x)$ approaches $\inf_{I}\rho$ with $x\in I$; for example, if $\rho(x)=1/(\pi\sqrt{x(1-x)})$, then $q=2$ if $\frac{1}{2}\in \overline{I}$ and $q=1$ otherwise. This work extends a result of Feng and Wei beyond the Gaussian potential.
\end{abstract}


%\small \renewcommand{\baselinestretch}{0.75}   \tableofcontents \addtocontents{toc}{\vspace{-0.2cm}} \normalsize \renewcommand{\baselinestretch}{1}

{
\hypersetup{linkcolor=black}

}

\section{Introduction}

We consider $n\times n$ random Hermitian matrices $H$ with law proportional to $e^{-n \tr V(H)}dH$, where $dH$ is the Lebesgue measure on the set of $n\times n$ Hermitian matrices, $V:\R \to \R\cup \{+\infty\}$ is the potential, and $\tr V(H) := \sum_{j=1}^{n}V(\lambda_{j})$, where $\lambda_{1}<\dots<\lambda_{n}$ are the eigenvalues of $H$. This measure is invariant under unitary conjugation and induces the following probability measure on the eigenvalues:
\begin{align}\label{def of eig distribution}
\frac{1}{Z_{n}}\prod_{1\leq i < j \leq n} (\lambda_{j}-\lambda_{i})^{2} \prod_{j=1}^{n}e^{-nV(\lambda_{j})}d\lambda_{j},
\end{align}
where $Z_{n}$ is the normalization constant, see e.g. \cite{Deift}. Under mild assumptions on $V$, the empirical measure $\frac{1}{n}\sum_{j=1}^{n}\delta_{\lambda_{j}}$ converges, as $n \to +\infty$, to an equilibrium measure $\mu$ with compact support $\mathcal{S}$. Classical examples of such ensembles include the Gaussian (GUE), Laguerre (LUE), and Jacobi (JUE) unitary ensembles. Let $\rho = d\mu/dx$ denote the equilibrium density. For the GUE, one has $V(x)=x^{2}/2$, $\mathcal{S}=[-2,2]$ and $\rho(x) = \frac{1}{2\pi}\sqrt{4-x^{2}}$.

\medskip Given a finite union of intervals $I$ lying in the bulk, i.e. such that $\overline{I} \subset \mathcal{S}^{\circ}$, let $\smash{m_{1}^{(n)}}>\smash{m_{2}^{(n)}}>\dots$ be the largest gaps of the form $\lambda_{i+1}-\lambda_{i}$ with $\lambda_{i},\lambda_{i+1}\in I$. In this work, we are interested in the asymptotic behavior of $\smash{m_{k}^{(n)}}$ as $n\to + \infty$ with $k$ fixed. 

\medskip The study of local fluctuations of the spectrum of large random Hermitian matrices has a long history \cite{Mehta, F book}. In the bulk, i.e. near a point of $\mathcal{S}^{\circ}$ where the density $\rho$ does not vanish, the typical gaps between eigenvalues are of order $n^{-1}$ and the statistics are described by the sine point process \cite{Dyson}. Near a soft edge, i.e. near a point of $\partial \mathcal{S}$ where $\rho$ vanishes as a square root, the typical gaps between eigenvalues are of order $n^{-2/3}$ and the statistics are described by the Airy point process \cite{ForresterAiryBessel, TW Airy}. Near a hard edge, i.e. near a point of $\partial \mathcal{S}$ where $\rho$ blows up as an inverse square root, the typical gaps between eigenvalues are of order $n^{-2}$ and the statistics are described by the Bessel point process \cite{ForresterAiryBessel, TW Bessel}. 

\medskip The study of \textit{extreme} gaps between bulk eigenvalues was initiated by Vinson \cite{Vinson2001}. For equilibrium measures supported on a single interval with two soft edges, he proved that the smallest spacing in the bulk of \eqref{def of eig distribution} is of order $n^{-4/3}$ as $n\to+\infty$, and he identified the corresponding limiting distribution. Vinson also provided heuristics suggesting that the largest gap $\smash{m_{1}^{(n)}}$ should be of order $\sqrt{\log n}/n$. Similar results for the circular unitary ensemble (CUE) were also obtained in \cite{Vinson2001}.

\medskip The literature on smallest gaps between bulk eigenvalues of random matrices is now quite rich. Using a different approach, Soshnikov \cite{Soshnikov} obtained the limiting distribution of the smallest gaps (between points on a growing interval) for determinantal point processes with translation-invariant kernels. Vinson's result was later extended for the $k$-th smallest gap (with $k$ fixed) in \cite{BAB2013} for the CUE and GUE, and in \cite{SJ2012} for a broader class of potentials $V$. It was also shown in \cite{BAB2013, SJ2012} that the locations of the smallest gaps are described by a Poisson point process. These results were subsequently generalized to various one-dimensional point processes, including the C$\beta$E \cite{FW2021} for $\beta\in \N_{>0}$, the GOE \cite{FTW2019}, the GSE \cite{FLY2024}, perturbed GUE matrices \cite{FG2016}, generalized Wigner matrices \cite{B2022, LLM2020, Zhang}, as well as stationary Gaussian processes \cite{FGY2024}. In dimension two, the first result is due to Shi and Jiang \cite{SJ2012}, who studied the smallest gaps between eigenvalues of Ginibre matrices. Other results on smallest gaps in dimension two include \cite{FY2023} on the zeros of Gaussian analytic functions, \cite{LM2024} on bulk eigenvalues of real Ginibre matrices, and \cite{C2025} on eigenvalues of random normal matrices.

\medskip Fewer results are available on the largest gaps between bulk eigenvalues of random matrices. For the GUE, Ben Arous and Bourgade \cite{BAB2013} established the following first order asymptotics: given an interval $I=[a,b]\subset (-2,2)$, $p>0$ and any sequence $l_{n} = n^{o(1)}$,
\begin{align}\label{lol21}
\Big( \inf_{x\in I} \sqrt{4-x^{2}} \Big) \frac{n m_{l_{n}}^{(n)}}{\sqrt{32 \log n}} \oset{L^{p}}{\underset{n\to \infty}{\xrightarrow{\hspace*{0.85cm}}}} 1.
\end{align}
The result \eqref{lol21} was later shown to remain valid for perturbed GUE matrices in \cite{FG2016}. 

\medskip More refined asymptotics were then obtained by Feng and Wei \cite{FW2018}. For $V(x) = x^{2}/2$, let $\smash{\tau_{k}^{(n)}}$ be the rescaled gaps defined via
\begin{align}\label{def of tauk}
m_{k}^{(n)} = \frac{1}{\inf_{x\in I}\sqrt{4-x^{2}}} \bigg( \frac{\sqrt{32 \log n}}{n} - \frac{5}{2} \frac{  \log(2\log n)}{n \sqrt{2\log n}} + \frac{4\tau_{k}^{(n)}}{n \sqrt{2\log n}} \bigg).
\end{align}
They proved that for any fixed $x\in \R$, the number of $\smash{\tau_{k}^{(n)}}$ falling in the interval $[x,+\infty)$ converges in distribution, as $n\to+\infty$,  to a Poisson random variable with mean $e^{c_{\star}-x}$, with
\begin{align}\label{def of cstar}
c_{\star} = c_{0} + M_{0}(I), \qquad M_{0}(I) = \begin{cases}
\frac{3}{2}\log(4-a^{2}) - \log(4|a|), & \mbox{if } a+b<0, \\
\frac{3}{2}\log(4-b^{2}) - \log(4|b|), & \mbox{if } a+b>0, \\
\frac{3}{2}\log(4-a^{2}) - \log(2|a|), & \mbox{if } a+b=0,
\end{cases}
\end{align}
and where $c_{0}=\frac{1}{12}\log 2 + 3 \zeta'(-1)$. Here, $\zeta$ is the Riemann zeta function. As a consequence, $\tau_{k}^{(n)}$ converges in distribution to a Gumbel random variable: for any fixed bounded interval $I_{1}\subset \R$,
\begin{align}\label{lol20}
\lim_{n\to + \infty} \PP(\tau_{k}^{(n)}\in I_{1}) = \int_{I_{1}} \frac{e^{k(c_{\star}-x)}}{(k-1)!}e^{-e^{c_{\star}-x}}dx.
\end{align}
(More precisely, only $\tau_{1}^{(n)}$ converges in distribution to a Gumbel random variable. For general $k\in \N_{>0}$, $\smash{\tau_{k}^{(n)}}$ converges in distribution to a generalized Gumbel random variable, referred to as a ``gamma–Gumbel" random variable in \cite{AMA2014}.)

The works \cite{BAB2013,FW2018} also contain results for the CUE that are analogous to \eqref{lol21} and \eqref{lol20}; for the CUE, the equilibrium measure is flat, and therefore the scaling \eqref{def of tauk} and the constant $c_{\star}$ are different (see also Remark \ref{remark:CUE} below). Using a comparison argument with the GUE, Bourgade \cite{B2022}, and later Landon, Lopatto and Marcinek \cite{LLM2020}, extended the convergence \eqref{lol20} to generalized Wigner matrices. In \cite{FM2026}, a Poisson approximation was established for the size and locations of the largest gaps between zeros of a stationary Gaussian process. In dimension two, the leading order behavior for the largest gap between bulk eigenvalues of complex Ginibre matrices was recently obtained in \cite{LO2025}. Closely related to the study of largest gaps is eigenvalue rigidity, which concerns the maximal deviation between eigenvalues and the quantiles of the equilibrium measure. For Wigner matrices, some of the first results were proved in \cite{EYY}. For the model \eqref{def of eig distribution}, rigidity results were obtained in \cite{CFLW2021, CFWW2025} in the presence of only soft edges, and in \cite{DL2025} for the JUE.

\medskip In this work, we generalize the GUE result \eqref{lol20} to the point process \eqref{def of eig distribution}. The equilibrium measure $\mu$ is defined as the unique minimizer of the functional
\begin{align*}
\sigma \mapsto \iint \log |x-y|^{-1}d\sigma(x)d\sigma(y) + \int V(x)d\sigma(x)
\end{align*}
among all Borel probability measures $\sigma$ on $\R$. It is known (see e.g. \cite{SaTo}) that $\mu$ is completely characterized by the conditions
\begin{align}
& V(x) + 2 \int_{\mathcal{S}} \log \frac{1}{|x-s|}d\mu(s) = \ell_{V}, & & x \in \mathcal{S}, \label{EL1} \\
& V(x) + 2 \int_{\mathcal{S}} \log \frac{1}{|x-s|}d\mu(s) \geq \ell_{V}, & & x \in \R\setminus \mathcal{S}, \label{EL2}
\end{align}
where $\ell_{V}\in \R$ is a constant. 
\begin{assumptions}\label{ass:V}\emph{Throughout, we assume the following:
\begin{itemize}
\item[(i)] $V:\R\to \R\cup\{+\infty\}$ satisfies the growth condition $V(x)/\log(1+x^{2}) \to + \infty$ as $|x|\to + \infty$. 
\item[(ii)] $V$ is \textit{regular}, meaning that the inequality \eqref{EL2} holds strictly, and that $\rho(x):=d\mu(x)/dx>0$ for every $x\in \mathcal{S}^{\circ}$.
\item[(iii)] $\rho$ is analytic in a neighborhood of $\mathcal{S}^{\circ}$, and $V|_{\mathcal{L}}$ is analytic in a neighborhood of $\mathcal{L}$, where $\mathcal{L}:=\{x\in \R:V(x)<+\infty\}$.
\item[(iv)] $V$ and $\mu$ are of one of the following three types:
\begin{enumerate}
\item[Type 1:] $\mathcal{S}$ is a finite union of disjoint intervals, each point of $\partial \mathcal{S}$ is a soft edge, and $\mathcal{L}=\R$.
\item[Type 2:] $\mathcal{S}$ is a single interval with one soft edge and one hard edge, $\mathcal{S}\subset \mathcal{L}$ and $\mathcal{L}$ is an unbounded interval satisfying $\partial \mathcal{L}\subset \partial \mathcal{S}$.
\item[Type 3:] $\mathcal{S}$ is a single interval with two hard edges, and $\mathcal{L}=\mathcal{S}$.
\end{enumerate}
\end{itemize}}
\end{assumptions}
Condition (i) ensures that \eqref{def of eig distribution} is well-defined. Condition (iv) allows $\mathcal{S}$ to have several components only when all boundary points are soft edges; this restriction is purely technical (see also the end of this section). Three canonical examples of potentials satisfying Assumptions \ref{ass:V}, corresponding respectively to types~1, 2, and~3, are provided by the GUE potential and by particular cases of the LUE and JUE potentials:
\begin{align}
& \mathrm{GUE}: & & V(x)=\frac{x^{2}}{2}, & & \mathcal{S}=[-2,2], & & \rho(x) = \frac{1}{2\pi}\sqrt{4-x^{2}}, \label{GUE} \\
& \mathrm{LUE}: & & V(x)=\begin{cases}
x, & \mbox{if } x\geq 0, \\
+\infty, & \mbox{if } x<0, 
\end{cases} & & \mathcal{S}=[0,4], & & \rho(x) = \frac{1}{2\pi}\sqrt{\frac{4-x}{x}}, \label{LUE} \\
& \mathrm{JUE}: & & V(x)=\begin{cases}
0, & \mbox{if } x\in [0,1], \\
+\infty, & \mbox{otherwise}, 
\end{cases} & & \mathcal{S}=[0,1], & & \rho(x) = \frac{1}{\pi\sqrt{x(1-x)}}. \label{JUE}
\end{align}
Let $I$ be a finite union of intervals lying in the bulk, i.e. satisfying $\overline{I} \subset \mathcal{S}^{\circ}$,
%, and suppose that $\overline{I}$ has pairwise distinct endpoints, that is,
%\begin{align*}
%\overline{I} = \cup_{j=1}^{p}[\mathfrak{a}_{j},\mathfrak{b}_{j}] \subset \mathcal{S}^{\circ} \quad \mbox{for some } p \in \N_{>0} \mbox{ and } \mathfrak{a}_{1}<\mathfrak{b}_{1}<\mathfrak{a}_{2} < \dots < \mathfrak{b}_{p}.
%\end{align*} 
and let $\mathcal{M} \subset \overline{I}$ be the set of minimizers of $\rho$ on $\overline{I}$, 
\begin{align}\label{def of Mcal}
\mathcal{M} := \{u \in \overline{I}: \rho(u) = \inf_{I}\rho\}
\end{align}
where $\inf_{I}\rho$ is shorthand notation for $\inf_{y\in I}\rho(y)$. By Assumptions \ref{ass:V}, $\rho$ is analytic and non-constant on $\mathcal{S}^{\circ}$. Hence, for each $u\in \mathcal{M}$, there exists $q_{u} \in \N_{>0}$ such that
\begin{align}\label{def of du}
\rho(x) = \inf_{I}\rho + d_{u} |x-u|^{q_{u}} + \bigO\big( |x-u|^{q_{u}+1} \big), \qquad \mbox{as } x\to u, \qquad d_{u} := \bigg|\frac{\rho^{(q_{u})}(u)}{q_{u}!}\bigg|>0.
\end{align}
If $u \in \overline{I}^{\circ}$, then $q_{u}$ is necessarily even.  Let 
\begin{align}\label{def of q}
q := \max \{q_{u} : u \in \mathcal{M}\}.
\end{align}
It turns out only the subset of $\mathcal{M}$ for which $q_{u}=q$ will contribute in the limiting distribution of the largest gaps $\smash{m_{k}^{(n)}}$. Hence, we define
\begin{align}\label{def of M star}
\mathcal{M}_{\star} := \{u \in \mathcal{M}: q_{u}=q\}.
\end{align}
Among the points of $\mathcal{M}_{\star}$, we need to distinguish those belonging to $\partial \overline{I}$ from those belonging to $\overline{I}^{\circ}$:
\begin{align}\label{def of A and B}
\mathcal{A} := \mathcal{M}_{\star} \cap \partial \overline{I}, \qquad \mathcal{B} := \mathcal{M}_{\star} \cap \overline{I}^{\circ}.
\end{align}
We now state our main result.
\begin{theorem}\label{thm:main}
Suppose $V:\R\to \R\cup\{+\infty\}$ is a potential satisfying Assumptions \ref{ass:V}, and let $I$ be a finite union of intervals such that $\overline{I} \subset \mathcal{S}^{\circ}$. Let $\lambda_{1}<\dots<\lambda_{n}$ be sampled according to \eqref{def of eig distribution}, and denote by $\smash{m_{k}^{(n)}}$ the $k$-th largest gap of the form $\lambda_{i+1}-\lambda_{i}$ with $\lambda_{i},\lambda_{i+1}\in I$.
We define the rescaled gaps $\smash{\tau_{k}^{(n)}}$ via
\begin{align}\label{def of taukn}
m_{k}^{(n)} = \frac{1}{2\pi \inf_{I}\rho} \bigg( \frac{\sqrt{32 \log n}}{n} + \frac{3q-8}{2q} \frac{  \log(2\log n)}{n \sqrt{2\log n}} + \frac{4\tau_{k}^{(n)}}{n \sqrt{2\log n}} \bigg),
\end{align}
where $q$ is as in \eqref{def of q}. For any fixed $x \in \R$, the random variables $\#\{\tau_{k}^{(n)}\in [x,+\infty)\}$ converge in distribution, as $n\to + \infty$, to a Poisson random variable with mean $e^{c_{V,I}-x}$, where
\begin{align}\label{def of CVI}
& c_{V,I} = c_{0} + \log \bigg[ \frac{\pi}{2} \frac{1}{q}\Gamma(\tfrac{1}{q}) \bigg( \sum_{u\in \mathcal{A}}  d_{u}^{-\frac{1}{q}} + 2 \sum_{u\in \mathcal{B}}  d_{u}^{-\frac{1}{q}} \bigg) (\inf_{I} \rho)^{1+\frac{1}{q}}  \bigg],
\end{align}
with $\mathcal{A}, \mathcal{B}$ as in \eqref{def of A and B} and $c_{0}=\frac{1}{12}\log 2 + 3 \zeta'(-1)$, and where $\Gamma(z)=\int_{0}^{+\infty}t^{z-1}e^{-t}dt$ is the standard Gamma function and $\zeta$ is the Riemann zeta function.

\medskip In particular, for any fixed $k \in \N_{>0}$ and fixed interval $I_{1}\subset \R$ that is either bounded or semi-infinite with upper endpoint $+\infty$,
\begin{align}\label{Gumbel main thm}
\lim_{n\to + \infty} \PP(\tau_{k}^{(n)}\in I_{1}) = \int_{I_{1}} \frac{e^{k(c_{V,I}-x)}}{(k-1)!}e^{-e^{c_{V,I}-x}}dx.
\end{align}
\end{theorem}
\begin{remark}\label{remark:GUE}\textbf{(Consistency check with GUE.)}
For the GUE potential $V(x)=x^{2}/2$, one has $\mathcal{S}=[-2,2]$ and $\rho(x) = \frac{1}{2\pi}\sqrt{4-x^{2}}$. If $I=[a,b]\subset (-2,2)$, there are three cases to consider: 
\begin{align}\label{lol22}
\begin{cases}
q=1, \; \mathcal{A}=\{a\}, \; \mathcal{B}=\emptyset,  \; \inf_{I}\rho = \frac{\sqrt{4-a^{2}}}{2\pi}, \; d_{a}^{-1}= \frac{2\pi\sqrt{4-a^{2}}}{|a|} & \mbox{if } a+b<0, \\
q=1, \; \mathcal{A}=\{b\}, \; \mathcal{B}=\emptyset, \; \inf_{I}\rho = \frac{\sqrt{4-b^{2}}}{2\pi}, \; d_{b}^{-1}= \frac{2\pi\sqrt{4-b^{2}}}{b} & \mbox{if } a+b>0, \\
q=1, \; \mathcal{A}=\{a,b\}, \; \mathcal{B}=\emptyset, \; \inf_{I}\rho = \frac{\sqrt{4-a^{2}}}{2\pi}, \; d_{a}^{-1}=d_{b}^{-1}= \frac{2\pi\sqrt{4-a^{2}}}{|a|} & \mbox{if } a+b=0.
\end{cases}
\end{align}
Substituting \eqref{lol22} in \eqref{def of CVI} yields $c_{V,I} = c_{\star}$, where $c_{\star}$ is given by \eqref{def of cstar}.
\end{remark}
Because of the importance of the LUE and the JUE, we briefly pause here to also specialize Theorem \ref{thm:main} to those cases (for simplicity, only when $I$ consists of a single interval).
\begin{remark}\label{remark:LUE and JUE}\textbf{(The LUE.)}
Recall from \eqref{LUE} that for the LUE, one has $\mathcal{S}=[0,4]$ and $\rho(x)=\frac{1}{2\pi}\sqrt{4-x}/\sqrt{x}$. Let $I=[a,b]\subset (0,4)$. Since $\rho$ is strictly decreasing on $\mathcal{S}$, only one case can occur: $q=1$, $\mathcal{A}=\{b\}$, $\mathcal{B}=\emptyset$, and thus
\begin{align*}
c_{V,I} = c_{0} + \log \bigg[ \frac{(4-b)^{3/2}\sqrt{b}}{8}\bigg].
\end{align*}
\end{remark}
\begin{remark}\label{remark:JUE}\textbf{(The JUE.)}
The situation is more complicated for the JUE. By \eqref{JUE}, one has $\mathcal{S}=[0,1]$ and $\rho(x) = 1/(\pi\sqrt{x(1-x)})$. Let $I=[a,b]\subset (0,1)$. Since $\frac{1}{2}$ is a global minimum of $\rho$, one needs to distinguish five cases: 
\begin{align*}
\begin{cases}
q=1, \; \mathcal{A} = \{b\}, \; \mathcal{B} = \emptyset, \; c_{V,I} = c_{0} + \log \frac{\sqrt{(1-b)b}}{1-2b}, & \mbox{if } b < \frac{1}{2}, \\
q=2, \; \mathcal{A} = \{\frac{1}{2}\}, \; \mathcal{B} = \emptyset, \; c_{V,I} = c_{0} + \log \frac{\sqrt{\pi}}{2\sqrt{2}}, & \mbox{if } b = \frac{1}{2}, \\
q=2, \; \mathcal{A} = \emptyset, \; \mathcal{B} = \{\frac{1}{2}\}, \; c_{V,I} = c_{0} + \log \frac{\sqrt{\pi}}{\sqrt{2}}, & \mbox{if } a < \frac{1}{2} < b, \\
q=2, \; \mathcal{A} = \{\frac{1}{2}\}, \; \mathcal{B} = \emptyset, \; c_{V,I} = c_{0} + \log \frac{\sqrt{\pi}}{2\sqrt{2}}, & \mbox{if } a = \frac{1}{2}, \\
q=1, \; \mathcal{A} = \{a\}, \; \mathcal{B} = \emptyset, \; c_{V,I} = c_{0} + \log \frac{\sqrt{(1-a)a}}{2a-1}, & \mbox{if } a > \frac{1}{2}.
\end{cases}
\end{align*}
\end{remark}

\begin{remark}\label{remark:CUE}\textbf{(Comparison with CUE.)}
Let $0 \leq \theta_{1} \leq \dots \leq \theta_{n} < 2\pi$ be the ordered eigenangles of an $n\times n$ Haar-distributed unitary matrix, and set $\theta_{n+1}=\theta_{1}+2\pi$. Let $\smash{m_{1}^{(n)}}>\smash{m_{2}^{(n)}}>\dots$ be the ordered largest gaps of the set $\{\theta_{i+1}-\theta_{i}: 1 \leq i \leq n\}$. First order asymptotics for $m_{k}^{(n)}$ as $n\to + \infty$ were obtained in \cite{BAB2013}. It was then shown in \cite{FW2018} that the rescaled gap $\smash{\tau_{k}^{(n)}}$ converges in distribution, as $n\to + \infty$ with $k$ fixed, to a gamma-Gumbel random variable, where $\smash{\tau_{k}^{(n)}}$ is defined by
\begin{align}\label{rescaled gaps of CUE}
m_{k}^{(n)} = \frac{\sqrt{32 \log n}}{n} + \frac{3}{2} \frac{\log(2 \log n)}{n\sqrt{2\log n}} + \frac{4\tau_{k}^{(n)}}{n\sqrt{2\log n}}.
\end{align}
In the setting of Theorem \ref{thm:main}, $\rho$ is analytic and non-constant on $\mathcal{S}^{\circ}$, which forces $q\in \N_{>0}$. By contrast, for the CUE, the equilibrium measure is flat, $d\mu(e^{i\theta}) = \frac{d\theta}{2\pi}$, and one can thus view the CUE as a $q=+\infty$ case. Although Theorem \ref{thm:main} does not apply to the CUE, it is interesting to note that by formally substituting $\rho\equiv \frac{1}{2\pi}$ in \eqref{def of taukn} and letting $q\to +\infty$, one recovers \eqref{rescaled gaps of CUE}.
\end{remark}

\begin{remark}\label{remark:heuristics}\textbf{(Heuristic derivation of Theorem \ref{thm:main}.)}
Let $\PP_{\mathrm{sine}}(r)$ be the probability that a given interval of size $r$ in the sine point process is free from points. By \cite{DIKZ2007}, 
\begin{align*}
\PP_{\mathrm{sine}}(r) = \exp \bigg( -\frac{\pi^{2}r^{2}}{8} - \frac{1}{4}\log \frac{\pi r}{2} + c_{0} + o(1) \bigg), \qquad \mbox{as } r\to + \infty.
\end{align*}
Let $\delta_{n}$ be a sequence converging sufficiently fast to $0$ as $n\to +\infty$. If $x\in \mathcal{S}^{\circ}$, one expects
\begin{align*}
\PP_{n}\Big( \lambda_{i} \notin [x,x+\delta_{n}], \; 1 \leq i \leq n \Big) \approx \PP_{\mathrm{sine}}(\delta_{n}n\rho(x)), \quad \mbox{as } n \to +\infty,
\end{align*}
where $\PP_{n}$ refers to \eqref{def of eig distribution}. Taking $\delta_{n} = \delta_{n}'\sqrt{32 \log n}/(2\pi \inf_{I} \rho \; n)$ with $\delta_{n}'$ of order $1$, we find
\begin{align*}
\PP_{n}\Big( \lambda_{i} \notin [x,x+\delta_{n}], \; 1 \leq i \leq n \Big) \approx \exp \bigg( - \tfrac{\delta_{n}'^{2}\rho(x)^{2}}{(\inf_{I}\rho)^{2}}\log n - \frac{1}{4}\log \big( \tfrac{\delta_{n}'\rho(x)}{\inf_{I}\rho}\sqrt{2\log n} \big) + c_{0} + o(1) \bigg)
\end{align*}
as $n\to+\infty$. This asymptotic formula shows that the largest gaps on $I$ are more likely to occur near points $u$ where $\rho(u)$ is minimal, i.e. near $u\in \mathcal{M}$ (recall \eqref{def of Mcal}). Let $\epsilon_{n}>0$ be small. Using \eqref{def of du}, 
\begin{align*}
& \int_{u-\epsilon_{n}}^{u+\epsilon_{n}} \PP_{n}\Big( \lambda_{i} \notin [x,x+\delta_{n}], \; 1 \leq i \leq n \Big)dx \approx 2 \int_{u}^{u+\epsilon_{n}} \PP_{n}\Big( \lambda_{i} \notin [x,x+\delta_{n}], \; 1 \leq i \leq n \Big)dx \\
& \approx 2 \int_{u}^{u+\epsilon_{n}} \exp \bigg( - \tfrac{\delta_{n}'^{2}\rho(x)^{2}}{(\inf_{I}\rho)^{2}}\log n - \frac{1}{4}\log \big( \tfrac{\delta_{n}'\rho(x)}{\inf_{I}\rho}\sqrt{2\log n} \big) + c_{0} + o(1) \bigg)dx \\
& \approx 2 \exp \bigg( - \delta_{n}'^{2}\log n - \frac{1}{4}\log \big( \delta_{n}'\sqrt{2\log n} \big) + c_{0} + o(1) \bigg) \int_{0}^{\epsilon_{n}} \exp \bigg( - \frac{2\delta_{n}'^{2} d_{u}}{\inf_{I}\rho}t^{q_{u}}\log n \bigg)dt.
\end{align*}
Choosing $\epsilon_{n} = M(\inf_{I}\rho)^{1/q_{u}}/(2\delta_{n}'^{2}d_{u}\log n)^{1/q_{u}}$ for some $M>0$ yields
\begin{align*}
\int_{0}^{\epsilon_{n}} \exp \bigg( - \frac{2\delta_{n}'^{2}d_{u}}{\inf_{I}\rho}t^{q_{u}}\log n \bigg)dt = \int_{0}^{M} \frac{(\inf_{I}\rho)^{1/q_{u}}e^{-x^{q_{u}}}dx}{(2\delta_{n}'^{2}d_{u}\log n)^{1/q_{u}}} = \frac{(\inf_{I}\rho)^{1/q_{u}}\big(\tfrac{1}{q_{u}}\Gamma(\tfrac{1}{q_{u}})+o(1)\big)}{(2\delta_{n}'^{2}d_{u}\log n)^{1/q_{u}}}
\end{align*}
as $M \to +\infty$. This computation indicates that the largest gaps are in fact most likely to occur near points $u\in \mathcal{M}_{\star}\subset \mathcal{M}$, where $\mathcal{M}_{\star}$ is defined by \eqref{def of M star}. Let $\mathcal{M}_{\star,n}$ be a $1/(\log n)^{1/q}$-neighborhood of $\mathcal{M}_{\star}$ such that $\mathcal{M}_{\star,n}\subset \overline{I}$, and define 
\begin{align}\label{lol42}
G = \{x\in I: \lambda_{i} \notin [x,x+\delta_{n}] \mbox{ for all } 1 \leq i \leq n \}. \qquad \mbox{Heuristically, } |G| \approx \sum_{i=1}^{k_{n}}(m_{i}^{(n)}-\delta_{n})_{+},
\end{align}
for some $k_{n}>0$ and where $x_{+}:=\max\{x,0\}$. The above analysis suggests that
\begin{align*}
\E[|G|] & \approx \int_{\mathcal{M}_{\star,n}} \PP_{n}\Big( \lambda_{i} \notin [x,x+\delta_{n}], \; 1 \leq i \leq n \Big)dx \\
& \approx \exp \bigg( - \delta_{n}'^{2}\log n - \frac{1}{4}\log \big( \delta_{n}'\sqrt{2\log n} \big) + c_{0} + o(1) \bigg) \frac{\tfrac{1}{q}\Gamma(\tfrac{1}{q}) \big( \sum_{u\in \mathcal{A}}  d_{u}^{-\frac{1}{q}} + 2 \sum_{u\in \mathcal{B}}  d_{u}^{-\frac{1}{q}}  \big)}{(2\delta_{n}'^{2}\log n)^{1/q}(\inf_{I}\rho)^{-1/q}}.
\end{align*}
We now choose $\delta_{n}' = 1 + \frac{3q-8}{2q} \frac{\log(2\log n)}{8 \log n} + \frac{4\tau}{8\log n}$ for some $\tau\in \R$, or equivalently
\begin{align*}
\delta_{n}=\delta_{n}(\tau) = \frac{1}{2\pi \inf_{I}\rho} \bigg( \frac{\sqrt{32 \log n}}{n} + \frac{3q-8}{2q} \frac{  \log(2\log n)}{n \sqrt{2\log n}} + \frac{4\tau}{n \sqrt{2\log n}} \bigg).
\end{align*}
With this choice, we obtain
\begin{align}\label{lol43}
\E[|G|] \approx e^{c_{0}-\tau}\frac{\tfrac{1}{q}\Gamma(\tfrac{1}{q}) \big( \sum_{u\in \mathcal{A}}  d_{u}^{-\frac{1}{q}} + 2 \sum_{u\in \mathcal{B}}  d_{u}^{-\frac{1}{q}}  \big)}{n\sqrt{2 \log n} (\inf_{I}\rho)^{-1/q}} = \frac{e^{c_{V,I}-\tau}}{n\sqrt{2 \log n}} \frac{2}{\pi \inf_{I}\rho}.
\end{align}
Define the rescaled set $\mathcal{T} := n\sqrt{2\log n}\frac{\pi \inf_{I}\rho}{2} G$. By \eqref{def of taukn}, \eqref{lol42} and \eqref{lol43},
\begin{align*}
|\mathcal{T}| \approx \sum_{i=1}^{k_{n}}(\tau_{i}^{(n)}-\tau)_{+}, \qquad \E[|\mathcal{T}|] \approx e^{c_{V,I}-\tau},
\end{align*}
and thus one expects
\begin{align*}
\E[\#\{\tau_{k}^{(n)}\in [\tau,+\infty)\}] \approx -\frac{d}{d \tau} e^{c_{V,I}-\tau} =  e^{c_{V,I}-\tau}.
\end{align*}
Assuming that the rescaled gaps $\tau_{k}^{(n)}$ behave asymptotically like the points of a Poisson process, the statement of Theorem \ref{thm:main} follows. The proof consists in making all these heuristics rigorous.
\end{remark}

\paragraph{Outline.} 
The proof of Theorem \ref{thm:main} follows the strategy developed in \cite{BAB2013, FW2018}, and relies crucially on the determinantal structure of the point process \eqref{def of eig distribution}. An important ingredient in \cite{BAB2013, FW2018} is a good control, at a scale slightly above local scale, of the difference between the GUE kernel and the sine kernel as $n\to + \infty$. In Section \ref{section:kernel}, we extend this result to general potentials $V$ satisfying Assumptions \ref{ass:V}, using the Riemann–Hilbert analysis from \cite{DKMVZ} for type~1 potentials and from \cite{CG2021} for types~2 and~3. (The exclusion of multi-cut regimes with hard edges in Assumptions \ref{ass:V} is precisely due to the lack of earlier works addressing those cases.)  In Section \ref{section:gap proba for CUE}, the estimates from Section \ref{section:kernel} are combined with results from \cite{DIKZ2007} on gap probabilities of the CUE to obtain new results on gap probabilities of \eqref{def of eig distribution} that are of size $\sqrt{\log n}/n$ as $n\to + \infty$. Section \ref{section:main section} adapts the argument of \cite{FW2018} from the GUE to general $V$ and completes the proof of Theorem \ref{thm:main}.

\section{Estimates on the kernel}\label{section:kernel}
Let $V:\R\to \R\cup\{+\infty\}$ be a potential satisfying Assumptions \ref{ass:V}. Let $\mu$ be the associated equilibrium measure, whose support is denoted $\mathcal{S}$ and whose density is $\rho := d\mu/dx$. Suppose $\lambda_{1}<\dots<\lambda_{n}$ are random variables sampled according to \eqref{def of eig distribution}, and consider the point process $\xi^{(n)} := \sum_{i=1}^{n}\delta_{\lambda_i}$. The $\ell$-point correlation functions $\{\rho_{\ell}:\R^{\ell} \to [0,+\infty)\}_{\ell \in \N_{>0}}$ of $\xi^{(n)}$ are defined such that
\begin{align}\label{def of DPP}
\mathbb{E}\bigg[\sum_{\substack{\lambda_{1},\ldots,\lambda_{\ell}\in \xi^{(n)} \\ \lambda_{i} \neq \lambda_{j} \; \mathrm{if} \; i \neq j}}f(x_{1},\ldots,x_{\ell})\bigg] = \int_{\R^{\ell}}f(x_1, \ldots, x_\ell)\rho_\ell(x_1, \ldots, x_\ell)dx_1 \ldots dx_\ell
\end{align}
holds for any measurable function $f:\R^{\ell}\to \R$ with compact support. As mentioned, $\xi^{(n)}$ is determinantal, which means (see e.g. \cite{SoshnikovSurvey}) that all correlation functions exist, and that
\begin{align}\label{def of rhok}
\rho_\ell(x_1, \ldots, x_\ell) & = \det(K_n(x_i,x_j))_{1\leq i, j\leq \ell}, \qquad \ell \in \N_{>0}, \; x_{1},\ldots,x_{\ell}\in \R.
\end{align}
Moreover, the correlation kernel $K_{n}:\R^{2}\to \R$ is given by
\begin{align}\label{def of Kn}
K_n(x,y) =  e^{-\frac{1}{2}nV(x)-\frac{1}{2}nV(y)}\sum_{j = 0}^{n-1}p_{j,n}(x)p_{j,n}(y),
\end{align}
where $p_{j,n}(x)=\kappa_{j,n}x^{j}+\ldots$ is the degree $j$ orthonormal polynomial satisfying $\kappa_{j,n}>0$ and
\begin{align}\label{ortho condition}
\int_{\R}p_{i,n}(x)p_{j,n}(x)e^{-nV(x)}dx = \delta_{i,j}, \qquad \mbox{for all } i,j \in \N.
\end{align}
The Christoffel-Darboux formula (see e.g. \cite{Deift}) allows to rewrite $K_{n}$ as
\begin{align}\label{Kn CD formula}
K_{n}(x,y) & = \frac{e^{-\frac{n}{2}(V(x)+V(y))}}{x-y}\frac{\kappa_{n-1,n}}{\kappa_{n,n}} \bigg[ p_{n,n}(x)p_{n-1,n}(y) - p_{n,n}(y)p_{n-1,n}(x) \bigg].
\end{align}
\begin{lemma}\label{lemma:kernel}
The following holds:
\begin{itemize}
\item[(i)] As $n\to + \infty$,
\begin{align}\label{rough estimate on Kn}
K_{n}(x,y) = \bigO(|x-y|^{-1}), 
\end{align}
uniformly for $x,y$ in compact subsets of $\mathcal{S}^{\circ}$.
\item[(ii)] As $n \to + \infty$,
\begin{align}\label{asymp Kn general V}
\frac{1}{n\rho(x_{0})}K_{n}\bigg(x_{0}+\frac{\xi}{n \rho(x_{0})},x_{0}+\frac{\eta}{n \rho(x_{0})}\bigg) = \frac{\sin\big( \pi(\xi - \eta) \big)}{\pi(\xi-\eta)} + \bigO\bigg( \frac{1+|\xi|+|\eta|}{n} \bigg),
\end{align}
uniformly for $x_{0}$ in compact subsets of $\mathcal{S}^{\circ}$ and for $\xi,\eta \in [-\log n, \log n]$. 
\item[(iii)] Suppose $\mathcal{S}=[a_{0},b_{0}]$ consists of a single interval. As $n \to + \infty$,
\begin{align}
& \frac{1}{n\rho(x_{0})}K_{n}\bigg(x_{0}+\frac{\xi}{n \rho(x_{0})},x_{0}+\frac{\eta}{n \rho(x_{0})}\bigg) = \frac{\sin\big( \pi(\xi - \eta) \big)}{\pi(\xi-\eta)} + \frac{1}{n} \bigg\{ \frac{\rho'(x_{0})}{2\rho(x_{0})^{2}}(\xi + \eta)\cos [\pi(\xi-\eta)] \nonumber \\
& - \frac{b_{0}-a_{0}}{4 \pi \rho(x_{0}) (b_{0}-x_{0})(x_{0}-a_{0})} \cos \bigg[ 2\pi n \int_{x_{0}}^{b_{0}}d\mu(x) - \pi (\xi + \eta) \bigg] \bigg\} + \bigO\bigg( \frac{1+|\xi|^{3}+|\eta|^{3}}{n^{2}} \bigg), \label{asymp Kn general V one cut higher order}
\end{align}
uniformly for $x_{0}$ in compact subsets of $\mathcal{S}^{\circ}=(a_{0},b_{0})$ and for $\xi,\eta \in [-\log n, \log n]$. 
\end{itemize}
\end{lemma}
\begin{remark}
Even for the GUE, \eqref{asymp Kn general V} slightly improves on the earlier result \cite[Lemma 3.4 (2) with $\delta_{n}=(\log n)/n$]{BAB2013}: our error term is $\bigO((\log n)/n)$, whereas \cite[Lemma 3.4]{BAB2013} yields $\bigO((\log n)^{2}/n)$. For the proof of Theorem \ref{thm:main}, only parts (i) and (ii) of Lemma \ref{lemma:kernel} are required. We nevertheless include part (iii), which provides an explicit subleading term in the one-cut regime and may be of independent interest; moreover, its proof requires little additional effort beyond that of (ii). Formula \eqref{asymp Kn general V one cut higher order} has been verified numerically for the GUE, LUE, and JUE. Our numerics further suggest that the error term in \eqref{asymp Kn general V one cut higher order} is optimal.
\end{remark}
\begin{proof}
Define $Y$ by 
\begin{equation}
Y(z) = \begin{pmatrix}\label{Y definition}
\kappa_{n,n}^{-1}p_{n,n}(z) & \frac{\kappa_{n,n}^{-1}}{2\pi i} \int_{\mathbb{R}} \frac{p_{n,n}(x)e^{-nV(x)}}{x-z}dx \\
-2\pi i \kappa_{n-1,n} p_{n-1,n}(z) & -\kappa_{n-1,n} \int_{\mathbb{R}} \frac{p_{n-1,n}(x)e^{-nV(x)}}{x-z}dx
\end{pmatrix}, \qquad z \in \C\setminus \R.
\end{equation}
Note that $Y_{11}$ and $Y_{21}$ are also well-defined on $\R$. Using \eqref{Kn CD formula} and \eqref{Y definition}, we can rewrite $K_{n}$ as
\begin{align}\label{Kn in terms of Y}
K_{n}(x,y) & = \frac{e^{-\frac{n}{2}(V(x)+V(y))}}{-2\pi i (x-y)} \bigg[ Y_{11}(x)Y_{21}(y) - Y_{11}(y)Y_{21}(x) \bigg].
\end{align}
The asymptotic analysis of $Y$ as $n\to + \infty$ was carried out in \cite{DKMVZ} for type 1 potentials (see also \cite{C2019, CFWW2025}), and in \cite{CG2021} for types 2 and 3 (in \cite{C2019,CG2021,CFWW2025}, our setting corresponds to $\alpha_{j}=\beta_{j}=0$ for all $j$, and the smooth functions $W$ in \cite{C2019,CG2021} and $\log F$ in \cite{CFWW2025} are identically $0$ here). Let $\epsilon>0$ be small but fixed, let $\mathcal{S}_{\epsilon}:=\{x\in \mathcal{S}:\mathrm{dist}(x,\partial \mathcal{S})\geq \epsilon\}$, and let $\mathcal{N}_{\epsilon}$ be small complex neighborhood of $\mathcal{S}_{\epsilon}$. For $z\in \mathcal{N}_{\epsilon}\setminus  \mathcal{S}$, we have (see e.g. \cite[(4.16), (4.22), (4.69)]{C2019} or \cite[(4.4), (4.13), (7.1)]{CFWW2025} for type 1, and \cite[(5.13), (5.20), (5.64)]{CG2021} for types 2 and 3)
\begin{align}
Y(z) & = e^{-\frac{n \ell_{V}}{2}\sigma_{3}}T(z)e^{ng(z)\sigma_{3}}e^{\frac{n \ell_{V}}{2}\sigma_{3}} \nonumber \\
& = e^{-\frac{n \ell_{V}}{2}\sigma_{3}}R(z)P^{(\infty)}(z) \left. \begin{cases}
\begin{pmatrix}
1 & 0 \\
e^{-2n\xi(z)} & 1
\end{pmatrix}, & \mbox{if } \im z >0 \\
\begin{pmatrix}
1 & 0 \\
-e^{-2n\xi(z)} & 1
\end{pmatrix}, & \mbox{if } \im z <0
\end{cases} \right\} e^{ng(z)\sigma_{3}}e^{\frac{n \ell_{V}}{2}\sigma_{3}}, \label{Y in terms of R and Pinf}
\end{align}
where $\ell_{V}$ is as in \eqref{EL1}--\eqref{EL2},
\begin{align*}
& g(z) = \int_{\mathcal{S}} \log(z-s)\rho(s)ds, \qquad \xi(z) = -\pi i \int_{\mathfrak{b}}^{z} \tilde{\rho}(s) ds = g(z) - \frac{V(z)}{2} + \frac{\ell_{V}}{2},
\end{align*}
the principal branch is used for the logarithms, and the path for the integral defining $\xi$ lies in $\C\setminus (-\infty,\mathfrak{b}]$, where $\mathfrak{b}$ is the right-most endpoint of $\partial \mathcal{S}$. The function $\tilde{\rho}$ is analytic in $U_{V}\setminus \mathcal{S}$, where $U_{V}$ is the domain of analyticity of $V$ ($\mathcal{L} \subset U_{V}$ by Assumptions \ref{ass:V}), and such that
\begin{align}\label{+ notation}
\tilde{\rho}_{+}(x) := \lim_{z\to x, \im z >0} \tilde{\rho}(z) = \rho(x), \qquad x \in \mathcal{S}.
\end{align}
The $\C^{2\times 2}$-valued function $P^{(\infty)}$ is analytic on $\C\setminus \mathcal{S}$. If $\mathcal{S}=[a_{0},b_{0}]$ consists of a single interval, then it is given by
\begin{align}\label{def of Pinf}
P^{(\infty)}(z) = \begin{pmatrix}
\frac{\gamma(z)+\gamma(z)^{-1}}{2} & \frac{\gamma(z)-\gamma(z)^{-1}}{2i} \\
\frac{\gamma(z)-\gamma(z)^{-1}}{-2i} & \frac{\gamma(z)+\gamma(z)^{-1}}{2}
\end{pmatrix}, \qquad \gamma(z) = \frac{(z-b_{0})^{1/4}}{(z-a_{0})^{1/4}},
\end{align}
where the principal branch is used for the roots. If $\mathcal{S}$ consists of several intervals, then $P^{(\infty)}(\cdot) = P^{(\infty)}(\cdot;n)$ depends on $n$, and its explicit expression is given in terms of $\theta$-functions (see e.g. \cite[(5.23) with $D(z) \equiv 1$]{CFWW2025}). In all cases, $P^{(\infty)}(z)$ is uniformly bounded on $\C$, except in small neighborhoods of $\partial \mathcal{S}$. The function $R(z)$ in \eqref{Y in terms of R and Pinf} is $\C^{2\times 2}$-valued, analytic in $\mathcal{N}_{\epsilon}\setminus \mathcal{S}$, and satisfies
\begin{align}\label{asymp of R}
R(z) = I + \bigO(n^{-1}), \qquad \mbox{as } n \to + \infty
\end{align}
uniformly for $z\in \mathcal{N}_{\epsilon}\setminus \mathcal{S}$ (see \cite[(4.74)]{C2019} and \cite[(7.4)]{CFWW2025} for type 1 potentials, and \cite[(5.67)]{CG2021} for types 2 and 3). Moreover, the determinants of $Y$, $R$ and $P^{(\infty)}$ are identically equal to 1 on $\C$.

For $x\in \mathcal{S}_{\epsilon}$, by taking the limit $z\to x$ with $\im z >0$ in \eqref{Y in terms of R and Pinf}, we obtain
\begin{align}
& Y_{11}(x) = e^{ng_{+}(x)} \begin{pmatrix}
1 & 0
\end{pmatrix} R_{+}(x)P^{(\infty)}_{+}(x) 
\begin{pmatrix}
1 \\
e^{-2n\xi_{+}(x)}
\end{pmatrix}, \label{Y11 in terms of R} \\
& Y_{21}(x) = e^{ng_{+}(x)+n \ell_{V}} \begin{pmatrix}
0 & 1
\end{pmatrix} R_{+}(x)P_{+}^{(\infty)}(x) 
\begin{pmatrix}
1 \\
e^{-2n\xi_{+}(x)}
\end{pmatrix}, \label{Y12 in terms of R}
\end{align}
where the subscript $+$ again means that the limit is taken from the upper half-plane (as in \eqref{+ notation}). Note that \eqref{EL1} can be rewritten as
\begin{align}\label{lol23}
g_{+}(x)+g_{-}(x) - V(x) + \ell_{V} = 0, \qquad x \in \mathcal{S},
\end{align}
(the subscript $-$ means that the limit is taken from the lower half-plane) and therefore 
\begin{align}\label{lol24}
\xi_{+}(x) = g_{+}(x) - \frac{V(z)}{2} + \frac{\ell_{V}}{2} = \frac{g_{+}(x)-g_{-}(x)}{2} = \pi i \int_{x}^{\mathfrak{b}}d\mu(s) \in i \R, \qquad x \in \mathcal{S}.
\end{align}
Thus $e^{-2n\xi_{+}(x)}$ remains bounded as $n\to + \infty$ with $x\in \mathcal{S}$, and it then follows from \eqref{Y11 in terms of R}--\eqref{Y12 in terms of R} that
\begin{align*}
Y_{11}(x) = \bigO(e^{ng_{+}(x)}), \qquad Y_{21}(x) = \bigO(e^{ng_{+}(x)+n\ell_{V}}), \qquad \mbox{as } n \to + \infty
\end{align*}
uniformly for $x\in \mathcal{S}_{\epsilon}$. Substituting the above in \eqref{Kn in terms of Y} yields
\begin{align*}
K_{n}(x,y) & = \frac{1}{x-y} \bigg[ \bigO(e^{ng_{+}(x)-\frac{n}{2}V(x)})\bigO(e^{ng_{+}(y)+n\ell_{V}-\frac{n}{2}V(y)}) \\
& + \bigO(e^{ng_{+}(y)-\frac{n}{2}V(y)})\bigO(e^{ng_{+}(x)+n\ell_{V}-\frac{n}{2}V(x)}) \bigg] = \bigO(|x-y|^{-1}), \qquad \mbox{as } n \to + \infty
\end{align*}
uniformly for $x,y\in \mathcal{S}_{\epsilon}$, where for the last equality we have again used \eqref{lol23} and \eqref{lol24}. This proves \eqref{rough estimate on Kn}. We now turn to the proof of \eqref{asymp Kn general V}. Using \eqref{Kn in terms of Y}, \eqref{asymp of R}, \eqref{Y11 in terms of R}, \eqref{Y12 in terms of R} and \eqref{lol24}, we obtain
\begin{align}\label{lol25}
-2\pi i (x-y)K_{n}(x,y) = \frac{Y_{11}(x)Y_{21}(y) - Y_{11}(y)Y_{21}(x)}{e^{\frac{n}{2}(V(x)+V(y))}} = \mathcal{T}(x,y;n) + \frac{\mathcal{R}(x,y;n)}{n},
\end{align}
where $\mathcal{R}(x,y;n)=\bigO(1)$ as $n\to + \infty$ uniformly for $x,y\in \mathcal{S}_{\epsilon}$, and 
\begin{align}
\mathcal{T}(x,y;n) & = e^{n(\xi_{+}(x)+\xi_{+}(y))}\Big( P^{(\infty)}_{11,+}(x)P^{(\infty)}_{21,+}(y) - P^{(\infty)}_{11,+}(y)P^{(\infty)}_{21,+}(x) \Big) \nonumber \\
& + e^{-n(\xi_{+}(x)+\xi_{+}(y))}\Big( P^{(\infty)}_{12,+}(x)P^{(\infty)}_{22,+}(y) - P^{(\infty)}_{12,+}(y)P^{(\infty)}_{22,+}(x) \Big) \nonumber \\
& + e^{n(\xi_{+}(x)-\xi_{+}(y))}\Big( P^{(\infty)}_{11,+}(x)P^{(\infty)}_{22,+}(y) - P^{(\infty)}_{12,+}(y)P^{(\infty)}_{21,+}(x) \Big) \nonumber \\
& - e^{n(\xi_{+}(y)-\xi_{+}(x))}\Big( P^{(\infty)}_{11,+}(y)P^{(\infty)}_{22,+}(x) - P^{(\infty)}_{12,+}(x)P^{(\infty)}_{21,+}(y) \Big). \label{def of Tcal}
\end{align}
In fact, since $\mathcal{T}(x,y;n)$ and the left hand-side of \eqref{lol25} are all $\bigO(x-y)$ as $x\to y$, we even have $\mathcal{R}(x,y;n)=\bigO(x-y)$ as $n\to + \infty$ uniformly for $x,y \in \mathcal{S}_{\epsilon}$.

Let $x_{0}\in \mathcal{S}^{\circ}$, and take
\begin{align*}
x_{n} = x_{0}+\frac{\xi}{n \rho(x_{0})}, \qquad y_{n} = x_{0}+\frac{\eta}{n \rho(x_{0})},
\end{align*}
where $\xi,\eta \in [-\log n, \log n]$. A direct computation, which uses the fact that $\det P^{(\infty)} \equiv 1$ and that $P^{(\infty)}_{+}(x)$ and all its derivatives are uniformly bounded as $n\to + \infty$ for $x\in \mathcal{S}_{\epsilon}$, yields
\begin{align*}
\frac{\mathcal{T}(x_{n},y_{n};n)}{-2\pi i n \rho(x_{0})(x_{n}-y_{n})} = \frac{\sin \big(\pi(\xi-\eta)\big)}{\pi(\xi-\eta)} + \bigO\bigg( \frac{1+|\xi| + |\eta|}{n} \bigg), \qquad \mbox{as } n \to + \infty
\end{align*}
uniformly for $x_{0}\in\mathcal{S}_{\epsilon}$ and for $\xi,\eta \in [-\log n, \log n]$. Substituting the above in \eqref{lol25} yields
\begin{align*}
\frac{1}{n \rho(x_{0})}K_{n}(x_{n},y_{n}) = \frac{\sin \big(\pi(\xi-\eta)\big)}{\pi(\xi-\eta)} + \bigO\bigg( \frac{1+|\xi| + |\eta|}{n} \bigg), \qquad \mbox{as } n \to + \infty
\end{align*}
uniformly for $x_{0}\in\mathcal{S}_{\epsilon}$ and for $\xi,\eta \in [-\log n, \log n]$, which is \eqref{asymp Kn general V}.

The above computation can easily be extended to the next order, using again results from \cite{C2019, CFWW2025, CG2021}. Instead of \eqref{asymp of R}, one needs to use
\begin{align}\label{asymp of R 2}
R(z) = I + \frac{R^{(1)}(z;n)}{n} + \bigO(n^{-2}), \qquad \mbox{as } n \to + \infty
\end{align}
uniformly for $z\in \mathcal{N}_{\epsilon}\setminus \mathcal{S}$. The function $R^{(1)}(z;n)$ was explicitly computed in \cite{C2019, CFWW2025, CG2021}, but all what we will need here is that $\tr R^{(1)}(z;n)\equiv 0$ (which directly follows from $\det R(z) \equiv 1$), and
\begin{align*}
R^{(1)}(z;n)=\bigO(1), \quad R^{(1)\prime}(z;n)=\bigO(1), \quad \mbox{as } n \to + \infty
\end{align*}
uniformly for $z\in \mathcal{N}_{\epsilon}\setminus \mathcal{S}$. Combining the above with \eqref{Y11 in terms of R}, \eqref{Y12 in terms of R}, and \eqref{lol24}, we find
\begin{align}\label{lol25 bis}
-2\pi i (x-y)K_{n}(x,y) = \mathcal{T}(x,y;n) + \frac{\mathcal{T}_{2}(x,y;n)}{n} + \frac{\mathcal{R}_{2}(x,y;n)}{n^{2}},
\end{align}
as $n\to + \infty$ uniformly for $x,y\in \mathcal{S}_{\epsilon}$. The expression for $\mathcal{T}_{2}$ can be made explicit in terms of the entries of $P^{(\infty)}$ and $R^{(1)}$, but since this expression is quite long, we omit it here. When $\mathcal{S}=[a_{0},b_{0}]$, the expression \eqref{def of Pinf} for $P^{(\infty)}$ simplifies the analysis, and we find
\begin{align*}
& \frac{\mathcal{T}(x_{n},y_{n};n)}{-2\pi i n \rho(x_{0})(x_{n}-y_{n})} = \frac{\sin \big(\pi(\xi-\eta)\big)}{\pi(\xi-\eta)} + \frac{1}{n} \bigg\{ \frac{\rho'(x_{0})}{2\rho(x_{0})^{2}}(\xi + \eta)\cos [\pi(\xi-\eta)] \nonumber \\
& \qquad - \frac{b_{0}-a_{0}}{4 \pi \rho(x_{0}) (b_{0}-x_{0})(x_{0}-a_{0})} \cos \bigg[ 2\pi n \int_{x_{0}}^{b_{0}}d\mu(x) - \pi (\xi + \eta) \bigg] \bigg\} + \bigO\bigg( \frac{1+|\xi|^{3}+|\eta|^{3}}{n^{2}} \bigg), \\
& \frac{\mathcal{T}_{2}(x_{n},y_{n};n)}{-2\pi i n^{2} \rho(x_{0})(x_{n}-y_{n})} = \bigO(n^{-2}), 
\end{align*}
as $n\to + \infty$ uniformly for $x_{0}\in\mathcal{S}_{\epsilon}$ and for $\xi,\eta \in [-\log n, \log n]$. Substituting the above in \eqref{lol25 bis} yields \eqref{asymp Kn general V one cut higher order}.
\end{proof}
\section{Gap probabilities of \eqref{def of eig distribution}}\label{section:gap proba for CUE}
In this section, we derive several results on gap probabilities of \eqref{def of eig distribution}. In particular, Lemma \ref{lemma: comparison with CUE} provides precise asymptotics for gap probabilities that are of size $\sqrt{\log n}/n$ as $n\to + \infty$. The proof relies on a comparison with the CUE, following the approach of \cite{BAB2013,FW2018}.

\subsection{Gap probabilities of the CUE}
Let $0 \leq \theta_{1} \leq \dots \leq \theta_{n} < 2\pi$ be the ordered eigenangles of an $n\times n$ Haar-distributed unitary matrix. The gap probability that no $\theta_{j}$'s lie in the interval $[0,2\alpha]$ for some $\alpha \in [0,\pi]$ is given by the following Toeplitz determinant 
\begin{align}\label{def of Dn alpha}
D_{n}(\alpha) = \det_{1 \leq j,k \leq n} \bigg( \frac{1}{2\pi}\int_{\alpha}^{2\pi-\alpha}e^{i(j-k)\theta}d\theta \bigg).
\end{align}
Let $\epsilon>0$ be fixed. By \cite{DIKZ2007}, there exists $s_{0}>0$ such that
\begin{align}\label{DIKZ}
\log D_{n}(\alpha) = n^{2} \log \cos \frac{\alpha}{2} - \frac{1}{4} \log \bigg( n \sin \frac{\alpha}{2} \bigg) + c_{0} + \bigO\bigg( \frac{1}{n \, \sin(\alpha/2)} \bigg), \qquad \mbox{as } n \to + \infty
\end{align}
uniformly for $\alpha \in [\frac{s_{0}}{n},\pi-\epsilon]$, where $c_{0} = \frac{1}{12}\log 2 + 3\zeta'(-1)$ and $\zeta$ is Riemann's zeta function.

\medskip As before, $V:\R\to \R\cup\{+\infty\}$ is a potential satisfying Assumptions \ref{ass:V}, $\mu$ is the associated equilibrium measure, $\mathcal{S}:=\mathrm{supp}(\mu)$, and $\rho := d\mu/dx$. Let $\lambda_{1}<\dots<\lambda_{n}$ be random variables sampled according to \eqref{def of eig distribution}. Given a finite union of intervals $I$ satisfying $\overline{I} \subset \mathcal{S}^{\circ}$, we let $\smash{m_{k}^{(n)}}$ be the $k$-th largest gap of the form $\lambda_{i+1}-\lambda_{i}$ with $\lambda_{i},\lambda_{i+1}\in I$. We also let $q$ be as in \eqref{def of q}. To prove Theorem \ref{thm:main}, the relevant scaling to consider is
\begin{align}\label{def of Gn}
G_{n}(x) = \frac{\sqrt{32 \log n}}{n} + \frac{3q-8}{2q} \frac{  \log(2\log n)}{n \sqrt{2\log n}} + \frac{4x}{n \sqrt{2\log n}}.
\end{align}
We define the rescaled gaps $\tau_{k}^{(n)}$ via
\begin{align}
& G_{n}(\tau_{k}^{(n)}) = S(I)m_{k}^{(n)}, \qquad S(I) = 2\pi \rho_{I},  \label{lol4}
\end{align}
where $\rho_{I}$ is shorthand notation for $\inf_{x\in I} \rho(x)$. It is easily checked that \eqref{lol4} is equivalent to \eqref{def of taukn}.

\medskip The two following lemmas are direct consequences of \eqref{DIKZ}. 
\begin{lemma}\label{lemma:lim of Dn}
As $n\to + \infty$, we have
\begin{align*}
D_{n}\bigg( \Big( 1+\frac{u}{\log n} \Big) \frac{G_{n}(x)}{2} \bigg) = \frac{(2\log n)^{\frac{1}{q}-\frac{1}{2}}}{n} e^{c_{0}-x-2u}\bigg( 1+\bigO\bigg( \frac{1}{\sqrt{\log n}} \bigg) \bigg),
\end{align*}
uniformly for $u$ and $x$ in compact subsets of $\R$.
\end{lemma}
\begin{proof}
Let $\alpha_{n} = ( 1+\frac{u}{\log n}) \frac{G_{n}(x)}{2}$. By \eqref{def of Gn}, we have $\alpha_{n}\asymp \sqrt{\log n}/n$ as $n\to +\infty$, so that \eqref{DIKZ} applies. Hence, as $n\to + \infty$,
\begin{align}
\log D_{n}(\alpha_{n}) & = n^{2} \log \cos \frac{\alpha_{n}}{2} - \frac{1}{4} \log \bigg( n \sin \frac{\alpha_{n}}{2} \bigg) + c_{0} + \bigO\bigg( \frac{1}{\sqrt{\log n}} \bigg) \nonumber \\
& = - \frac{n^{2}\alpha_{n}^{2}}{8} -\frac{1}{4} \log \bigg( \frac{n\alpha_{n}}{2} \bigg)  + c_{0} + \bigO\bigg( \frac{1}{\sqrt{\log n}} \bigg) \label{lol26} \\
& = -\log n + \frac{2-q}{2q}\log \big( 2\log n \big) - x -2u + c_{0} + \bigO\bigg( \frac{1}{\sqrt{\log n}} \bigg), \nonumber
\end{align}
uniformly for $x,u$ in compact subsets of $\R$, which is the claim.
\end{proof}

\begin{lemma}\label{lemma:bound}
For every fixed $x\in \R$ and $A>1$, there exists $n_{0}=n_{0}(x,A)>0$ such that 
\begin{align}
D_{n}\bigg( w \frac{G_{n}(x)}{2} \bigg) \leq e^{-(w-1)\log n+1} D_{n}\bigg( \frac{G_{n}(x)}{2} \bigg)
\end{align}
holds for all $n\geq n_{0}$ and $w\in [1,A]$.
\end{lemma}
\begin{proof}
Let $\alpha_{n} = \frac{G_{n}(x)}{2}$. By \eqref{def of Gn}, $w\alpha_{n}\asymp \sqrt{\log n}/n$ as $n\to +\infty$ uniformly for $w\in [1,A]$, so \eqref{DIKZ} applies. Hence, using \eqref{lol26} with $u=0$ and with $\alpha_{n}$ replaced by $w\alpha_{n}$, we find
\begin{align}
\log D_{n}(w\alpha_{n}) & = - \frac{n^{2}w^{2}\alpha_{n}^{2}}{8} -\frac{1}{4} \log \bigg( \frac{nw\alpha_{n}}{2} \bigg)  + c_{0} + \bigO\bigg( \frac{1}{\sqrt{\log n}} \bigg), \label{lol26 bis}
\end{align}
as $n\to + \infty$ uniformly for $w\in [1,A]$. Hence,
\begin{align}
\log \frac{D_{n}(w\alpha_{n})}{D_{n}(\alpha_{n})} & = - (w^{2}-1)\frac{n^{2}\alpha_{n}^{2}}{8} -\frac{1}{4} \log w + \bigO\bigg( \frac{1}{\sqrt{\log n}} \bigg) \nonumber \\
& = (w^{2}-1) \bigg( -\log n + \frac{8-3q}{8q}\log(2\log n) - x \bigg)  -\frac{1}{4} \log w + \bigO\bigg( \frac{1}{\sqrt{\log n}} \bigg) \label{lol27}
\end{align}
as $n\to + \infty$ uniformly for $w\in [1,A]$. Let $n_{0,1}(A,x)$ be such that the $\bigO$-term in \eqref{lol27} is $\leq 1$ for all $n \geq n_{0,1}(A,x)$, and let $n_{0,2}$ be such that 
\begin{align*}
& (w^{2}-1) \bigg( -\log n + \frac{8-3q}{8q}\log(2\log n) - x \bigg)  -\frac{1}{4} \log w \\
& = -(w-1)\log n - w(w-1) \log n + (w^{2}-1) \bigg( \frac{8-3q}{8q}\log(2\log n) - x \bigg)  -\frac{1}{4} \log w \\
& \leq -(w-1)\log n + (w-1) \bigg\{ - \log n + (1+A) \bigg( \frac{5}{8}\log(2\log n) - x \bigg) \bigg\}
\end{align*}
holds for all $n \geq n_{0,2}$ and $w\in [1,A]$. Finally, let $n_{0,3}(A,x)$ be such that $ (1+A) \big( \frac{5}{8}\log(2\log n) - x \big) \leq \log n$ for all $n \geq n_{0,3}(A,x)$. The claim now directly follows with $n_{0}(A,x)=\max\{n_{0,1}(A,x),$ $n_{0,2},n_{0,3}(A,x)\}$.
\end{proof}

As in proof of Lemma \ref{lemma:kernel}, given $\epsilon>0$, we let $\mathcal{S}_{\epsilon}:=\{x\in \mathcal{S}:\mathrm{dist}(x,\partial \mathcal{S})\geq \epsilon\}$. In what follows, $\PP_{n}$ refers to \eqref{def of eig distribution}, and $\PP_{n}^{\mathrm{CUE}}$ refers to the CUE probability measure given by
\begin{align*}
\frac{1}{(2\pi)^{n}} \prod_{1 \leq j < k \leq n} |e^{i\theta_{j}}-e^{i\theta_{k}}|^{2} \prod_{j=1}^{n}d\theta_{j}, \qquad 0 \leq \theta_{1} \leq \ldots \leq \theta_{n} < 2\pi.
\end{align*}

\begin{lemma}\label{lemma: comparison with CUE}
Let $\epsilon >0$ be small but fixed. As $n\to + \infty$,
\begin{align*}
\PP_{n}\bigg( \lambda_{i} \notin [x,x+\tfrac{\delta_{n}}{\rho(x)}], \; 1 \leq i \leq n \bigg) = \PP_{n}^{\mathrm{CUE}}\bigg( \theta_{i}\notin [0,2\pi \delta_{n}], \; 1 \leq i \leq n \bigg)+ \bigO\bigg( \frac{1}{n(\log n)^{100}} \bigg),
\end{align*}
uniformly for $x\in \mathcal{S}_{\epsilon}$ and for $n \delta_{n}/\sqrt{\log n}$ in compact subsets of $(0,+\infty)$.
\end{lemma}
%  and for 
%\begin{align*}
%\delta_{n} \in \bigg( \frac{c_{*}\sqrt{\log n}}{n}, \frac{C_{0}\sqrt{\log n}}{n} \bigg),
%\end{align*}
%we have, for all large enough $n$,
\begin{proof}
The proof is similar to \cite[Lemma 3.5]{BAB2013} and \cite[Lemma 12]{FW2018}, but relies on Lemma \ref{lemma:kernel} instead of \cite[Lemma 3.4]{BAB2013}. Let $A,B$ be the integral operators with kernels given by
\begin{align}
& A(u,v)=-\frac{1}{n\rho(x)}K_{n}\bigg(x+\frac{u}{n \rho(x)},x+\frac{v}{n \rho(x)}\bigg) \chi_{[0,n\delta_{n}]}(u)\chi_{[0,n\delta_{n}]}(v), \label{kernel A} \\
& B(u,v) = -\frac{2\pi}{n}K_{n}^{\mathrm{CUE}}\bigg( \frac{2\pi}{n}u,\frac{2\pi}{n}v \bigg)\chi_{[0,n\delta_{n}]}(u)\chi_{[0,n\delta_{n}]}(v), \label{kernel B}
\end{align}
where $K_{n}$ is given by \eqref{def of Kn}, $K_{n}^{\mathrm{CUE}}(x,y):=\frac{1}{2\pi}\frac{\sin(\frac{n}{2}(x-y))}{\sin(\frac{1}{2}(x-y))}$, and $\chi_{J}$ denotes the indicator function of an interval $J$, so that (with $\det$ denoting the Fredholm determinant)
\begin{align*}
\PP_{n}\bigg( \lambda_{i} \notin [x,x+\tfrac{\delta_{n}}{\rho(x)}], \; 1 \leq i \leq n \bigg) = \det(I+A), \quad \PP_{n}^{\mathrm{CUE}}\bigg( \theta_{i}\notin [0,2\pi \delta_{n}], \; 1 \leq i \leq n \bigg)=\det(I+B).
\end{align*}
We will use the following inequality from \cite[Lemma 6]{FW2018}:
\begin{align}
& |\det(I+A) - \det(I+B)| \leq \big| \exp \big( -\tr(B-A)(I+B)^{-1} \big) - 1 \big| \; \det(I+B). \label{lol28}
\end{align}
By \eqref{def of Dn alpha}, $\det(I+B) = D_{n}(\pi \delta_{n})$. Since $\delta_{n} \asymp \sqrt{\log n}/n$ as $n\to + \infty$, we can use \eqref{lol26} (with $\alpha_{n}=\pi \delta_{n}$), so that
\begin{align*}
\det(I+B) = \exp \bigg[ - \frac{n^{2}\pi^{2}\delta_{n}^{2}}{8} - \frac{1}{4}\log \bigg( \frac{n\pi \delta_{n}}{2} \bigg) + c_{0} + \bigO \bigg( \frac{1}{\sqrt{\log n}} \bigg) \bigg] \qquad \mbox{as } n \to + \infty.
\end{align*}
Hence, for all sufficiently large $n$,
\begin{align}\label{lol29}
(\log n)^{-\frac{1}{7}}n^{-\frac{\pi^{2}}{8}\frac{n^{2} \delta_{n}^{2}}{\log n}} \leq \det(I+B) \leq n^{-\frac{\pi^{2}}{8}\frac{n^{2} \delta_{n}^{2}}{\log n}}.
\end{align}
It remains to evaluate the trace in \eqref{lol28}. Using again \cite[Lemma 6]{FW2018}, we have
\begin{align}
& |\tr (B-A)(I+B)^{-1}| \leq |\tr(A-B)| + |A-B|_{2}|B|_{2} \| (I+B)^{-1} \|. \label{upper bound Tr resolvant}
\end{align}
It follows from \cite[proof of Lemma 3.5]{BAB2013} that $|B|_{2}^{2} = \bigO((\log n)^{2/3})$ as $n\to + \infty$, and from the definition of $K_{n}^{\mathrm{CUE}}$ that $\mathrm{Tr}(B) = \int_{0}^{n\delta_{n}}B(u,u)du = -n\delta_{n}$. Moreover, by Lemma \ref{lemma:kernel} (ii),
\begin{align}\label{estimate Tr A and B}
\mathrm{Tr}(A) = \int_{0}^{n\delta_{n}} A(u,u)du = -n\delta_{n} + \bigO(n \delta_{n}^{2}) = -n\delta_{n} + \bigO\bigg( \frac{\log n}{n} \bigg), \qquad \mbox{as } n \to + \infty
\end{align}
uniformly for $x\in \mathcal{S}_{\epsilon}$. By \cite[Lemma 7]{FW2018}, we have
\begin{align}\label{upper bound norm resolvant pre}
\| (I+B)^{-1} \| \leq \frac{e^{1+\tr B}}{\det(I+B)}.
\end{align}
Since $\mathrm{Tr}(B) = -n\delta_{n}$, it follows from \eqref{lol29} and \eqref{upper bound norm resolvant pre} that, for all sufficiently large $n$,
\begin{align}\label{upper bound norm resolvant}
\| (I+B)^{-1} \| = \bigO\bigg( (\log n)^{-200-\frac{4}{3}}n^{\frac{\pi^{2}}{8}\frac{n^{2} \delta_{n}^{2}}{\log n}} \bigg).
\end{align}
By Lemma \ref{lemma:kernel} (ii), the infinite norm between the two kernels \eqref{kernel A}--\eqref{kernel B} is $\bigO(\delta_{n}) = \bigO(\sqrt{\log n}/n)$ as $n\to + \infty$ uniformly for $x\in \mathcal{S}_{\epsilon}$, and thus
\begin{align}\label{estimate for AmB2}
|A-B|_{2} = \bigg( \int_{0}^{n\delta_{n}}\int_{0}^{n\delta_{n}} |(A-B)(u,v)|_{2}^{2}dudv \bigg)^{1/2} = \bigO(n\delta_{n}^{2}) = \bigO \bigg( \frac{\log n}{n} \bigg), \quad \mbox{as } n \to + \infty
\end{align}
uniformly for $x\in \mathcal{S}_{\epsilon}$. Substituting $\mathrm{Tr}(B) = -n\delta_{n}$, \eqref{estimate Tr A and B}, \eqref{estimate for AmB2}, $|B|_{2} = \bigO((\log n)^{1/3})$, and \eqref{upper bound norm resolvant} in \eqref{upper bound Tr resolvant} yields 
\begin{align}
& |\tr (B-A)(I+B)^{-1}| = \bigO \bigg( \frac{\log n}{n} + (\log n)^{-200}n^{\frac{\pi^{2}}{8}\frac{n^{2} \delta_{n}^{2}}{\log n}-1} \bigg). 
\end{align}
The rest of the proof is divided into three cases.

\medskip \noindent \underline{Case 1}: $\frac{\pi^{2}}{8}\frac{n^{2} \delta_{n}^{2}}{\log n}-1 \leq 100 \frac{\log \log n}{\log n}$ holds for all sufficiently large $n$. 

Then $|\tr (B-A)(I+B)^{-1}| \to 0$ as $n\to + \infty$, and therefore, using \eqref{lol28} and \eqref{lol29},
\begin{align*}
& |\det(I+A) - \det(I+B)| = \bigO \Big( \big|\tr(B-A)(I+B)^{-1} \big| \; \det(I+B) \Big) = \bigO\bigg( \frac{(\log n)^{-200}}{n} \bigg)
\end{align*}
as $n\to + \infty$ uniformly for $x\in \mathcal{S}_{\epsilon}$ and for $n \delta_{n}/\sqrt{\log n}$ in compact subsets of $(0,+\infty)$. 

\medskip \noindent \underline{Case 2}: $\frac{\pi^{2}}{8}\frac{n^{2} \delta_{n}^{2}}{\log n}-1 \geq 100 \frac{\log \log n}{\log n}$ holds for all sufficiently large $n$. 

Then \eqref{lol29} implies $\det(I+B) \leq n^{-1}(\log n)^{-100}$. Let $\tilde{\delta}_{n}$ be such that $\frac{\pi^{2}}{8}\frac{n^{2} \tilde{\delta}_{n}^{2}}{\log n}-1 = 100 \frac{\log \log n}{\log n}$. Using case 1 with $\delta_{n}$ replaced by $\tilde{\delta}_{n}$, we obtain
\begin{align*}
\det(I+A) & = \PP_{n}\bigg( \lambda_{i} \notin [x,x+\tfrac{\delta_{n}}{\rho(x)}], \; 1 \leq i \leq n \bigg) 
 \leq \PP_{n}\bigg( \lambda_{i} \notin [x,x+\tfrac{\tilde{\delta}_{n}}{\rho(x)}], \; 1 \leq i \leq n \bigg) \\
& = \PP_{n}^{\mathrm{CUE}}\bigg( \theta_{i}\notin [0,2\pi \tilde{\delta}_{n}], \; 1 \leq i \leq n \bigg) + \bigO\bigg( n^{-1}(\log n)^{-200} \bigg) \\
& = \bigO\Big( n^{-1}(\log n)^{-100} \Big),
\end{align*}
as $n\to + \infty$ uniformly for $x\in \mathcal{S}_{\epsilon}$ and for $n \delta_{n}/\sqrt{\log n}$ in compact subsets of $(0,+\infty)$. Hence,
\begin{align*}
& |\det(I+A) - \det(I+B)| = \bigO\bigg( \frac{(\log n)^{-100}}{n} \bigg) \qquad \mbox{as } n \to + \infty
\end{align*}
uniformly for $x\in \mathcal{S}_{\epsilon}$ and for $n \delta_{n}/\sqrt{\log n}$ in compact subsets of $(0,+\infty)$.

\medskip \noindent \underline{General case}: $\frac{\pi^{2}}{8}\frac{n^{2} \delta_{n}^{2}}{\log n}$ lies in a fixed compact subset of $(0,+\infty)$ for all $n$. This case directly follows from cases 1 and 2 (otherwise one could construct a subsequence contradicting either case 1 or 2).
\end{proof} 

We will also use the following result, proved in \cite{BAB2013} for the CUE and the GUE; the proof carries over to \eqref{def of eig distribution} without change and is therefore omitted. Recall that $\mathcal{L}=\{x\in \R:V(x)<+\infty\}$.
\begin{lemma}\label{lemma:negative correlation}
Let $I_{1}$ and $I_{2}$ be disjoint compact subsets of $\mathcal{L}$. Then
\begin{align*}
\PP(\xi^{(n)}(I_{1}\cup I_{2})=0) \leq \PP(\xi^{(n)}(I_{1})=0)\PP(\xi^{(n)}(I_{2})=0).
\end{align*}
\end{lemma}

Recall that the sets $\mathcal{A}$ and $\mathcal{B}$ are defined in \eqref{def of A and B}. For each $u\in \mathcal{A}\cup \mathcal{B}$, let $d_{u}>0$ be as in \eqref{def of du}. We define
\begin{align}\label{def of MI}
M(I) = \bigg( \sum_{u\in \mathcal{A}}  d_{u}^{-\frac{1}{q}} + 2 \sum_{u\in \mathcal{B}}  d_{u}^{-\frac{1}{q}} \bigg)  \rho_{I}^{\frac{1}{q}}\frac{1}{q}\Gamma(\tfrac{1}{q}),
\end{align}
where we recall that $\rho_{I} = \inf_{x\in I} \rho(x)$. We finish this section with a technical lemma. 
\begin{lemma}\label{lemma:big lemma}
Fix $x \in \R$. As $n\to + \infty$, 
\begin{align*}
\int_{I}D_{n}\bigg( \frac{\rho(y)}{\rho_{I}} \frac{G_{n}(x)}{2}  \bigg) dy = \frac{M(I)}{n\sqrt{2\log n}}e^{c_{0}-x}(1+o(1)).
\end{align*}
\end{lemma}
\begin{proof}We divide the proof in three cases.

\medskip \noindent \underline{Case 1.} Suppose first that $I$ consists of a single interval and that $\rho'(y)>0$ for all $y\in I^{\circ}$. Let $a$ and $b$ be the left and right endpoints of $I$. We write
\begin{align*}
\mathcal{I}:=\int_{a}^{b}D_{n}\bigg( \frac{\rho(y)}{\rho_{I}} \frac{G_{n}(x)}{2}  \bigg) dy = \int_{a}^{\rho^{-1}(\rho(a)(1+\frac{1}{\sqrt{\log n}}))}D_{n}\bigg( \frac{\rho(y)}{\rho_{I}} \frac{G_{n}(x)}{2}  \bigg) dy + \mathcal{I}_{3},
\end{align*}
where 
\begin{align*}
\mathcal{I}_{3} := \int_{\rho^{-1}(\rho(a)(1+\frac{1}{\sqrt{\log n}}))}^{b}D_{n}\bigg( \frac{\rho(y)}{\rho_{I}} \frac{G_{n}(x)}{2} \bigg) dy.
\end{align*}
Using $\rho(a)=\rho_{I}$ and Lemma \ref{lemma:bound}, we get, for all sufficiently large $n$,
\begin{align*}
\mathcal{I}_{3} \leq e \; D_{n} \bigg( \frac{G_{n}(x)}{2} \bigg) \int_{\rho^{-1}(\rho(a)(1+\frac{1}{\sqrt{\log n}}))}^{b} e^{-(\frac{\rho(y)}{\rho(a)}-1)\log n} dy.
\end{align*}
Since $\rho$ is increasing on $I$, we have
\begin{align*}
\mathcal{I}_{3} \leq e \; D_{n} \bigg( \frac{G_{n}(x)}{2} \bigg) e^{-\sqrt{\log n}} (b-a), \qquad \mbox{for } n \mbox{ large}.
\end{align*}
By Lemma \ref{lemma:lim of Dn} (with $u=0$), this implies $\mathcal{I}_{3} = \bigO((\log n)^{\frac{1}{q}-\frac{1}{2}}e^{-\sqrt{\log n}}/n)$ as $n\to + \infty$. In particular, $\mathcal{I}_{3} = \bigO(n^{-1} (\log n)^{-100})$ as $n\to + \infty$.

To estimate $\mathcal{I}-\mathcal{I}_{3}$, we use $\rho(a)=\rho_{I}$ and the change of variables $\rho(y)=t$ and then $\frac{t}{\rho_{I}} = 1+\frac{u}{\log n}$ to get
\begin{align*}
& \mathcal{I}-\mathcal{I}_{3} = \int_{\rho(a)}^{\rho(a)(1+\frac{1}{\sqrt{\log n}})} D_{n}\bigg( \frac{t}{\rho_{I}} \frac{G_{n}(x)}{2}  \bigg) (\rho^{-1})'(t)dt \\
& = \frac{\rho_{I}}{\log n} \int_{0}^{\sqrt{\log n}} D_{n}\bigg( \bigg[ 1+\frac{u}{\log n} \bigg] \frac{G_{n}(x)}{2}  \bigg) (\rho^{-1})'\Big(\rho(a)\big[1+\tfrac{u}{\log n}\big]\Big)du.
\end{align*}
Let $M>0$ be a large but fixed constant. We write $\mathcal{I}-\mathcal{I}_{3}= \mathcal{I}_{1} + \mathcal{I}_{2}$, where
\begin{align*}
& \mathcal{I}_{1} := \frac{\rho_{I}}{\log n} \int_{0}^{M} D_{n}\bigg( \bigg[ 1+\frac{u}{\log n} \bigg] \frac{G_{n}(x)}{2}  \bigg) (\rho^{-1})'\Big(\rho(a)\big[1+\tfrac{u}{\log n}\big]\Big)du, \\
& \mathcal{I}_{2} := \frac{\rho_{I}}{\log n} \int_{M}^{ \sqrt{\log n}} D_{n}\bigg( \bigg[ 1+\frac{u}{\log n} \bigg] \frac{G_{n}(x)}{2}  \bigg) (\rho^{-1})'\Big(\rho(a)\big[1+\tfrac{u}{\log n}\big]\Big)du.
\end{align*}
Using again Lemma \ref{lemma:bound}, for large enough $n$ we get
\begin{align}\label{lol6}
\mathcal{I}_{2} \leq \frac{\rho_{I}}{\log n}e \; D_{n}\bigg( \frac{G_{n}(x)}{2}  \bigg) \int_{M}^{\sqrt{\log n}} e^{-u} (\rho^{-1})'\Big(\rho(a)\big[1+\tfrac{u}{\log n}\big]\Big)du.
\end{align}
Since $\rho(y) = \rho(a) + d_{a}(y-a)^{q} + \bigO((y-a)^{q+1})$ as $y\to a$ for some $d_{a}>0$, we have
\begin{align}\label{expansion of rho inv}
\rho^{-1}(s) = a + \bigg( \frac{s-\rho(a)}{d_{a}} \bigg)^{\frac{1}{q}} \Big( 1 + \bigO \big( (s-\rho(a))^{\frac{1}{q}} \big) \Big), \qquad \mbox{as } s\to \rho(a)_{+}.
\end{align}
Moreover, since $\rho$ is analytic on $I$, the expansion \eqref{expansion of rho inv} can be differentiated, that is,
\begin{align}\label{expansion of rho inv der}
(\rho^{-1})'(s) = \frac{1}{d_{a}^{1/q}q}( s-\rho(a) )^{\frac{1}{q}-1} \Big( 1 + \bigO \big( (s-\rho(a))^{\frac{1}{q}} \big) \Big), \qquad \mbox{as } s\to \rho(a)_{+}.
\end{align}
Using \eqref{expansion of rho inv der} with $s=\rho(a)(1+\frac{u}{\log n})$ yields
\begin{align}\label{lol7}
(\rho^{-1})'\Big(\rho(a)(1+\tfrac{u}{\log n})\Big) = \frac{1}{d_{a}^{1/q}q} \bigg( \frac{\rho(a)u}{\log n} \bigg)^{\frac{1}{q}-1} \bigg( 1 + \bigO \Big( \frac{u^{1/q}}{(\log n)^{1/q}} \Big) \bigg), \qquad \mbox{as } n\to +\infty
\end{align}
uniformly for $u\in [0,\sqrt{\log n}]$. Using the above in \eqref{lol6}, we find
\begin{align}
\mathcal{I}_{2} & \leq \frac{\rho_{I}}{\log n}e \; D_{n}\bigg( \frac{G_{n}(x)}{2}  \bigg) \int_{M}^{\sqrt{\log n}} e^{-u} \frac{2}{d_{a}^{1/q}q} \bigg( \frac{\rho(a)u}{\log n} \bigg)^{\frac{1}{q}-1}du \nonumber \\
& = \frac{2e \, \rho(a)^{\frac{1}{q}-1}}{d_{a}^{1/q}q}\frac{\rho_{I}}{(\log n)^{\frac{1}{q}}}D_{n}\bigg( \frac{G_{n}(x)}{2}  \bigg) \int_{M}^{\sqrt{\log n}} e^{-u}  u^{\frac{1}{q}-1}du \\
& \leq \frac{e^{-M}}{(\log n)^{\frac{1}{q}}} \frac{2e \, \rho_{I}^{\frac{1}{q}}}{d_{a}^{1/q}q} D_{n}\bigg( \frac{G_{n}(x)}{2}  \bigg) \int_{0}^{+\infty} e^{-u}  u^{\frac{1}{q}-1}du \nonumber \\
& \leq \frac{C e^{-M}}{n\sqrt{\log n}} \label{lol8}
\end{align}
for some $C>0$ and for all large enough $n$, and where for the last inequality we have used Lemma \ref{lemma:lim of Dn} (with $u=0$). Finally, for $\mathcal{I}_{1}$, using again \eqref{lol7} and Lemma \ref{lemma:lim of Dn}, as $n\to + \infty$ we obtain
\begin{align}
\mathcal{I}_{1} & = \frac{\rho_{I}}{\log n} \int_{0}^{M} D_{n}\bigg( \bigg[ 1+\frac{u}{\log n} \bigg] \frac{G_{n}(x)}{2}  \bigg) (\rho^{-1})'\Big(\rho(a)\big[1+\tfrac{u}{\log n}\big]\Big)du \nonumber \\
& = \frac{\rho_{I}}{\log n} \int_{0}^{M} e^{c_{0}-x-2u}\frac{(2\log n)^{\frac{1}{q}-\frac{1}{2}}}{n}(1+o(1)) \frac{1}{d_{a}^{1/q}q} \bigg( \frac{\rho(a)u}{\log n}\bigg)^{\frac{1}{q}-1} ( 1 + o (1) )du \nonumber \\
& = \frac{(\rho_{I}/d_{a})^{\frac{1}{q}}}{n\sqrt{2\log n}} \frac{2^{\frac{1}{q}}}{q}e^{c_{0}-x}(1+o(1)) \int_{0}^{M}u^{\frac{1}{q}-1}e^{-2u}du. \label{lol9}
\end{align}
Since $M$ can be taken arbitrarily large, \eqref{lol8} and \eqref{lol9} imply that 
\begin{align*}
\mathcal{I}_{1} + \mathcal{I}_{2} & = \frac{(\rho_{I}/d_{a})^{\frac{1}{q}}}{n\sqrt{2\log n}} \frac{2^{\frac{1}{q}}}{q}e^{c_{0}-x}(1+o(1)) \int_{0}^{+\infty}u^{\frac{1}{q}-1}e^{-2u}du \\
& = \frac{(\rho_{I}/d_{a})^{\frac{1}{q}}}{n\sqrt{2\log n}} \frac{1}{q}\Gamma(\tfrac{1}{q})e^{c_{0}-x}(1+o(1)), \qquad \mbox{as } n\to + \infty.
\end{align*}

\medskip \noindent \underline{Case 2.} Suppose now that $I$ consists of a single interval and that $\rho'(y)<0$ for all $y\in I^{\circ}$. The proof is similar to case 1. We write
\begin{align*}
\mathcal{I}:=\int_{a}^{b}D_{n}\bigg( \frac{\rho(y)}{\rho_{I}} \frac{G_{n}(x)}{2}  \bigg) dy = \int_{\rho^{-1}(\rho(b)(1+\frac{1}{\sqrt{\log n}}))}^{b}D_{n}\bigg( \frac{\rho(y)}{\rho_{I}} \frac{G_{n}(x)}{2}  \bigg) dy + \mathcal{I}_{3},
\end{align*}
where 
\begin{align*}
\mathcal{I}_{3} := \int_{a}^{\rho^{-1}(\rho(b)(1+\frac{1}{\sqrt{\log n}}))} D_{n}\bigg( \frac{\rho(y)}{\rho_{I}} \frac{G_{n}(x)}{2} \bigg) dy.
\end{align*}
Using $\rho(b)=\rho_{I}$ and Lemma \ref{lemma:bound}, we get, for all sufficiently large $n$,
\begin{align*}
\mathcal{I}_{3} \leq e \; D_{n} \bigg( \frac{G_{n}(x)}{2} \bigg) \int_{a}^{\rho^{-1}(\rho(b)(1+\frac{1}{\sqrt{\log n}}))} e^{-(\frac{\rho(y)}{\rho(b)}-1)\log n} dy \leq e \; D_{n} \bigg( \frac{G_{n}(x)}{2} \bigg) e^{-\sqrt{\log n}} (b-a),
\end{align*}
which implies, by Lemma \ref{lemma:lim of Dn} (with $u=0$), that $\mathcal{I}_{3} = \bigO(n^{-1} (\log n)^{-100})$ as $n\to + \infty$.

Using the change of variables $\rho(y)=t$ and $\frac{t}{\rho_{I}} = 1+\frac{u}{\log n}$, we obtain
\begin{align*}
& \mathcal{I}-\mathcal{I}_{3} = \int_{\rho(b)}^{\rho(b)(1+\frac{1}{\sqrt{\log n}})} D_{n}\bigg( \frac{t}{\rho_{I}} \frac{G_{n}(x)}{2}  \bigg) (-1)(\rho^{-1})'(t)dt \\
& = \frac{\rho_{I}}{\log n} \int_{0}^{ \sqrt{\log n}} D_{n}\bigg( \bigg[ 1+\frac{u}{\log n} \bigg] \frac{G_{n}(x)}{2}  \bigg) (-1)(\rho^{-1})'\Big(\rho(b)\big[1+\tfrac{u}{\log n}\big]\Big)du = \mathcal{I}_{1} + \mathcal{I}_{2},
\end{align*}
where we used that $\rho(b)=\rho_{I}$, and 
\begin{align*}
& \mathcal{I}_{1} := \frac{\rho_{I}}{\log n} \int_{0}^{M} D_{n}\bigg( \bigg[ 1+\frac{u}{\log n} \bigg] \frac{G_{n}(x)}{2}  \bigg) (-1)(\rho^{-1})'\Big(\rho(b)\big[1+\tfrac{u}{\log n}\big]\Big)du, \\
& \mathcal{I}_{2} := \frac{\rho_{I}}{\log n} \int_{M}^{ \sqrt{\log n}} D_{n}\bigg( \bigg[ 1+\frac{u}{\log n} \bigg] \frac{G_{n}(x)}{2}  \bigg) (-1)(\rho^{-1})'\Big(\rho(b)\big[1+\tfrac{u}{\log n}\big]\Big)du.
\end{align*}
By Lemma \ref{lemma:bound}, for large enough $n$,
\begin{align}\label{lol6 bis}
\mathcal{I}_{2} \leq \frac{\rho_{I}}{\log n}e \; D_{n}\bigg( \frac{G_{n}(x)}{2}  \bigg) \int_{M}^{\sqrt{\log n}} e^{-u} (-1)(\rho^{-1})'\Big(\rho(b)\big[1+\tfrac{u}{\log n}\big]\Big)du.
\end{align}
Since $\rho(y) = \rho(b) + d_{b}(b-y)^{q} + \bigO((b-y)^{q+1})$ as $y\to b$ for some $d_{b}>0$, we have
\begin{align}\label{expansion of rho inv bis}
\rho^{-1}(s) = b - \bigg( \frac{s-\rho(b)}{d_{b}} \bigg)^{\frac{1}{q}} \Big( 1 + \bigO \big( (s-\rho(b))^{\frac{1}{q}} \big) \Big), \qquad \mbox{as } s\to \rho(b)_{+}.
\end{align}
Since $\rho$ is analytic on $I$, we also have
\begin{align}\label{expansion of rho inv der bis}
(\rho^{-1})'(s) = \frac{-1}{d_{b}^{1/q}q}( s-\rho(b) )^{\frac{1}{q}-1} \Big( 1 + \bigO \big( (s-\rho(b))^{\frac{1}{q}} \big) \Big), \qquad \mbox{as } s\to \rho(b)_{+}.
\end{align}
Using \eqref{expansion of rho inv der bis} with $s=\rho(b)(1+\frac{u}{\log n})$ yields
\begin{align}\label{lol7 bis}
(\rho^{-1})'\Big(\rho(b)(1+\tfrac{u}{\log n})\Big) = \frac{-1}{d_{b}^{1/q}q} \bigg( \frac{\rho(b)u}{\log n} \bigg)^{\frac{1}{q}-1} \bigg( 1 + \bigO \Big( \frac{u}{(\log n)^{1/q}} \Big) \bigg), \qquad \mbox{as } n\to +\infty
\end{align}
uniformly for $u\in [0,\sqrt{\log n}]$. Using the above in \eqref{lol6 bis}, we find (in a similar way as \eqref{lol8})
\begin{align}
\mathcal{I}_{2} & \leq \frac{C e^{-M}}{n\sqrt{\log n}} \label{lol8 bis}
\end{align}
for some $C>0$ and for all large enough $n$. Finally, for $\mathcal{I}_{1}$, using \eqref{lol7 bis} and Lemma \ref{lemma:lim of Dn}, we find
\begin{align}
\mathcal{I}_{1} & = \frac{\rho_{I}}{\log n} \int_{0}^{M} D_{n}\bigg( \bigg[ 1+\frac{u}{\log n} \bigg] \frac{G_{n}(x)}{2}  \bigg) (-1)(\rho^{-1})'\Big(\rho(b)\big[1+\tfrac{u}{\log n}\big]\Big)du \nonumber \\
& = \frac{\rho_{I}}{\log n} \int_{0}^{M} e^{c_{0}-x-2u}\frac{(2\log n)^{\frac{1}{q}-\frac{1}{2}}}{n}(1+o(1)) \frac{1}{d_{b}^{1/q}q} \bigg( \frac{\rho(b)u}{\log n}\bigg)^{\frac{1}{q}-1} ( 1 + o (1) )du \nonumber \\
& = \frac{(\rho_{I}/d_{b})^{\frac{1}{q}}}{n\sqrt{2\log n}} \frac{2^{\frac{1}{q}}}{q}e^{c_{0}-x}(1+o(1)) \int_{0}^{M}u^{\frac{1}{q}-1}e^{-2u}du, \qquad \mbox{as } n \to + \infty. \label{lol9 bis}
\end{align}
Since $M$ can be taken arbitrarily large, \eqref{lol8 bis} and \eqref{lol9 bis} imply as $n\to + \infty$ that
\begin{align*}
\mathcal{I}_{1} + \mathcal{I}_{2} & = \frac{(\rho_{I}/d_{b})^{\frac{1}{q}}}{n\sqrt{2\log n}} \frac{2^{\frac{1}{q}}}{q}e^{c_{0}-x}(1+o(1)) \int_{0}^{+\infty}u^{\frac{1}{q}-1}e^{-2u}du = \frac{(\rho_{I}/d_{b})^{\frac{1}{q}}}{n\sqrt{2\log n}} \frac{1}{q}\Gamma(\tfrac{1}{q})e^{c_{0}-x}(1+o(1)).
\end{align*}
\underline{General case.} Suppose now that $I$ consists in a finite union of intervals. Since $\rho$ is analytic and non-constant on $\mathcal{S}^{\circ}$, we can write $I = \cup_{j=1}^{m_{1}}I^{(j)} \cup \cup_{j=1}^{m_{2}} J^{(j)}$ for some $m_{1},m_{2}\in \N$, where $\{I^{(j)},J^{(j)}\}$ are intervals on which $\rho$ is monotone and such that $\rho_{I} = \rho_{I^{(j)}}$ for each $j\in \{1,\ldots,m_{1}\}$ and $\rho_{I} < \rho_{J^{(j)}}$ for each $j\in \{1,\ldots,m_{2}\}$. Define
\begin{align*}
\mathcal{J}^{(j)} = \int_{J^{(j)}}D_{n}\bigg( \frac{\rho(y)}{\rho_{I}} \frac{G_{n}(x)}{2}  \bigg) dy, \qquad j=1,\ldots,m_{2}.
\end{align*}
Using Lemma \ref{lemma:bound}, we get, for any $j\in \{1,\ldots,m_{2}\}$ and all sufficiently large $n$,
\begin{align*}
\mathcal{J}^{(j)} \leq e \; D_{n} \bigg( \frac{G_{n}(x)}{2} \bigg) \int_{J^{(j)}} e^{-(\frac{\rho(y)}{\rho_{I}}-1)\log n} dy \leq e \; D_{n} \bigg( \frac{G_{n}(x)}{2} \bigg) \; |J^{(j)}| \; e^{-(\frac{\rho_{J^{(j)}}}{\rho_{I}}-1)\log n}.
\end{align*}
By Lemma \ref{lemma:lim of Dn} (with $u=0$), this implies \begin{align}\label{lol10}
\mathcal{J}^{(j)} = \bigO\Big( n^{-1-\frac{1}{2}(\frac{\rho_{J^{(j)}}}{\rho_{I}}-1)} \Big), \qquad \mbox{as } n \to + \infty, \; j\in \{1,\ldots,m_{2}\}.
\end{align}
For $j\in \{1,\ldots,m_{1}\}$, let us write $a^{(j)}$ and $b^{(j)}$ for the left and right endpoints of $I^{(j)}$. Thus
\begin{align*}
I^{(j)} = (a^{(j)},b^{(j)}) \quad \mbox{or} \quad I^{(j)} = [a^{(j)},b^{(j)}) \quad \mbox{or} \quad I^{(j)} = (a^{(j)},b^{(j)}] \quad \mbox{or} \quad I^{(j)} = [a^{(j)},b^{(j)}].
\end{align*} 
Suppose there exists $j\in \{1,\ldots,m_{1}\}$ such that
\begin{align*}
\rho(y) = \rho_{I} + \mathfrak{c}_{j}(y-a^{(j)})^{q_{j}} + \bigO((y-a^{(j)})^{q_{j}+1}), \qquad \mbox{as } y\to a^{(j)}
\end{align*}
for some $\mathfrak{c}_{j}>0$ and some $q_{j}\in \N_{>0}$ with $q_{j}<q$. Then, by adapting the proof of Case 1, we find
\begin{align}\label{lol11}
I^{(j)} = \frac{1}{(2\log n)^{\frac{1}{q_{j}}-\frac{1}{q}}}  \frac{(\rho_{I}/\mathfrak{c}_{j})^{\frac{1}{q_{j}}}}{n\sqrt{2\log n}} \frac{1}{q_{j}}\Gamma(\tfrac{1}{q_{j}})e^{c_{0}-x}(1+o(1)) = o \bigg( \frac{1}{n\sqrt{\log n}} \bigg), \qquad \mbox{as } n\to + \infty.
\end{align}
Similarly, if $j\in \{1,\ldots,m_{1}\}$ is such that $\rho(y) = \rho_{I} + \mathfrak{c}_{j}(b^{(j)}-y)^{q_{j}} + \bigO((b^{(j)}-y)^{q_{j}+1})$ for some $\mathfrak{c}_{j}>0$ and some $q_{j}\in \N_{>0}$ with $q_{j}<q$, then, by adapting the proof of Case 2, we find $I^{(j)} = o(n^{-1}(\log n)^{-1/2})$ as $n\to + \infty$.

Suppose now that $j\in \{1,\ldots,m_{1}\}$ is such that 
\begin{align*}
\rho(y) = \rho_{I} + \mathfrak{c}_{j}(y-a^{(j)})^{q_{j}} + \bigO((y-a^{(j)})^{q+1}), \qquad \mbox{as } y\to a^{(j)}
\end{align*}
for some $\mathfrak{c}_{j}>0$ with $q_{j}=q$. Then, with the same argument as for Case 1, we find
\begin{align}\label{lol12}
I^{(j)} = \frac{(\rho_{I}/\mathfrak{c}_{j})^{\frac{1}{q}}}{n\sqrt{2\log n}} \frac{1}{q}\Gamma(\tfrac{1}{q})e^{c_{0}-x}(1+o(1)), \qquad \mbox{as } n\to + \infty.
\end{align}
Similarly, if $j\in \{1,\ldots,m_{1}\}$ is such that 
\begin{align*}
\rho(y) = \rho_{I} + \mathfrak{c}_{j}(b^{(j)}-y)^{q_{j}} + \bigO((b^{(j)}-y)^{q+1}), \qquad \mbox{as } y\to b^{(j)}
\end{align*}
for some $\mathfrak{c}_{j}>0$ with $q_{j}=q$, then the proof of Case 2 also yields \eqref{lol12}. 
Since
\begin{align*}
\int_{I}D_{n}\bigg( \frac{\rho(y)}{S(I)} \frac{G_{n}(x)}{2}  \bigg) dy = \sum_{j=1}^{m_{1}}\mathcal{I}^{(j)} + \sum_{j=1}^{m_{2}} \mathcal{J}^{(j)},
\end{align*}
the claim follows by combining \eqref{lol10}, \eqref{lol11}, \eqref{lol12}, and the fact that
\begin{align}\label{mathfrak c to du}
\sum_{\substack{j=1 \\ q_{j}=q}}^{m_{1}}\mathfrak{c}_{j}^{-1/q} = \sum_{u\in \mathcal{A}}  d_{u}^{-\frac{1}{q}} + 2 \sum_{u\in \mathcal{B}}  d_{u}^{-\frac{1}{q}}.
\end{align}
\end{proof}

\section{Proof of Theorem \ref{thm:main}}\label{section:main section}
In this section, we complete the proof of Theorem \ref{thm:main} by following the method developed in \cite{FW2018}. 

Recall that $V:\R\to \R\cup\{+\infty\}$ is a potential satisfying Assumptions \ref{ass:V}, $\mu$ is the associated equilibrium measure, $\mathcal{S}:=\mathrm{supp}(\mu)$, $\rho := d\mu/dx$, and $\lambda_{1}<\dots<\lambda_{n}$ are random variables sampled according to \eqref{def of eig distribution}. As mentioned in Section \ref{section:kernel}, the point process $\xi^{(n)} = \sum_{i=1}^{n}\delta_{\lambda_{i}}$ is determinantal. Let $I$ be a finite union of interval satisfying $\overline{I}\subset \mathcal{S}^{\circ}$, and define
\begin{align*}
M_{0}(I) = \log \bigg(\frac{M(I)S(I)}{4}\bigg),
\end{align*}
where $M(I)$ is given \eqref{def of MI} and we recall that $S(I) = 2\pi \inf_{x\in I} \rho(x)$. It is easily verified that $c_{V,I} = c_{0}+M_{0}(I)$, where $c_{V,I}$ is as in \eqref{def of CVI}. 

%Define
%\begin{align*}
%f(x) = e^{c_{V,I}-x} = \frac{M(I)S(I)}{4} e^{c_{0}-x}.
%\end{align*}
%Then we have $-f'(x)=f''(x)=e^{c_{2}-x}$. 

Let $\Lambda(I) = \{i : \lambda_{i},\lambda_{i+1}\in I\}$, and define $k_{n}=\# \Lambda(I)$. By \cite[Lemma 1]{FW2018}, to prove Theorem \ref{thm:main}, it suffices to show that 
\begin{align}\label{what we want to prove}
\lim_{n\to + \infty} \E \sum_{\substack{i_{1},\ldots,i_{k} \in \{1,\ldots,k_{n}\} \\  \mbox{ \footnotesize all distinct}}} \prod_{j=1}^{k}(\tau_{i_{j}}^{(n)}-x_{j})_{+} = \prod_{j=1}^{k}e^{c_{V,I}-x_{j}} = \prod_{j=1}^{k}\frac{M(I)S(I)}{4}e^{c_{0}-x_{j}}
\end{align}
holds for every $k\in \N_{>0}$ and $x_{1},\ldots,x_{k}\in \R$.

\bigskip For $y>0$ and $k\in \{1,\ldots,n-1\}$, let
\begin{align*}
J_{k}(y) := \{x \in \R: [x,x+y]\subset (\lambda_{k},\lambda_{k+1})\} = \begin{cases}
(\lambda_{k},\lambda_{k+1}-y), & \mbox{if } y < \lambda_{k+1}-\lambda_{k}, \\
\emptyset, & \mbox{if } y \geq \lambda_{k+1}-\lambda_{k}.
\end{cases}
\end{align*}
In particular,
\begin{align*}
|J_{k}(y)| = (\lambda_{k+1}-\lambda_{k}-y)_{+}, \quad J_{k}(y) \subset (\lambda_{k},\lambda_{k+1}), \quad J_{k}(y) \cap J_{l}(y) = \emptyset \mbox{ for } k \neq l,
\end{align*}
where $x_{+}:=\max\{x,0\}$. For $a_{1},\ldots,a_{k}\in(0,+\infty)$, define
\begin{align}\label{def of Sigmak}
\Sigma_{k}(a_{1},\ldots,a_{k}) = \bigcup_{\substack{i_{1},\ldots,i_{k}\in \Lambda(I) \\
\mathrm{all} \, \mathrm{distinct}}} \prod_{j=1}^{k} J_{i_{j}}(a_{j}) \subset (I^{\circ})^{k}.
\end{align}
Since $\Sigma_{k}$ is a union of disjoint boxes,
\begin{align*}
\big|\Sigma_{k}(a_{1},\ldots,a_{k})\big| = \sum_{\substack{i_{1},\ldots,i_{k}\in \Lambda(I) \\
\mathrm{all} \, \mathrm{distinct}}} \prod_{j=1}^{k} (\lambda_{i_{j}+1}-\lambda_{i_{j}}-a_{j})_{+} = \sum_{\substack{i_{1},\ldots,i_{k}\in \{1,\ldots,k_{n}\} \\
\mathrm{all} \, \mathrm{distinct}}} \prod_{j=1}^{k} (m_{i_{j}}^{(n)}-a_{j})_{+}.
\end{align*}
By \eqref{def of Gn}--\eqref{lol4}, we have
\begin{align*}
\tau_{k}^{(n)} - x = \big( G_{n}(\tau_{k}^{(n)}) - G_{n}(x) \big) \frac{n}{4}\sqrt{2\log n} = \big( S(I)m_{k}^{(n)} - G_{n}(x) \big) \frac{n}{4}\sqrt{2\log n},
\end{align*}
and thus, for any $x_{1},\ldots,x_{k}\in \R$,
\begin{align*}
\sum_{\substack{i_{1},\ldots,i_{k}\in \{1,\ldots,k_{n}\} \\
\mathrm{all} \, \mathrm{distinct}}} \prod_{j=1}^{k} (\tau_{i_{j}}^{(n)}-x_{j})_{+} & = \bigg( \frac{nS(I)}{4}\sqrt{2\log n} \bigg)^{k} \sum_{\substack{i_{1},\ldots,i_{k}\in \{1,\ldots,k_{n}\} \\
\mathrm{all} \, \mathrm{distinct}}} \prod_{j=1}^{k} \bigg(m_{i_{j}}^{(n)}-\frac{G_{n}(x_{j})}{S(I)}\bigg)_{+} \\
& = \bigg( \frac{nS(I)}{4}\sqrt{2\log n} \bigg)^{k} \big| \Sigma_{k}\big(\tfrac{G_{n}(x_{1})}{S(I)},\ldots,\tfrac{G_{n}(x_{k})}{S(I)}\big)\big|.
\end{align*}
Hence 
\begin{align*}
& \E \sum_{\substack{i_{1},\ldots,i_{k}\in \{1,\ldots,k_{n}\} \\
\mathrm{all} \, \mathrm{distinct}}} \prod_{j=1}^{k} (\tau_{i_{j}}^{(n)}-x_{j})_{+}  = \bigg( \frac{nS(I)}{4}\sqrt{2\log n} \bigg)^{k} \E \big| \Sigma_{k}\big(\tfrac{G_{n}(x_{1})}{S(I)},\ldots,\tfrac{G_{n}(x_{k})}{S(I)}\big)\big| \\
& = \bigg( \frac{nS(I)}{4}\sqrt{2\log n} \bigg)^{k} \int_{\R^{n}} \sum_{\substack{i_{1},\ldots,i_{k}\in \Lambda(I) \\
\mathrm{all} \, \mathrm{distinct}}} \prod_{j=1}^{k} (\lambda_{i_{j}+1}-\lambda_{i_{j}}-\tfrac{G_{n}(x_{j})	}{S(I)})_{+} d\PP_{n}(\lambda_{1},\ldots,\lambda_{n}) \\
& = \bigg( \frac{nS(I)}{4}\sqrt{2\log n} \bigg)^{k} \int_{\R^{n}} \sum_{\substack{i_{1},\ldots,i_{k}\in \Lambda(I) \\
\mathrm{all} \, \mathrm{distinct}}} \prod_{j=1}^{k} \int_{I} \mathbf{1}_{y_{j}\in J_{i_{j}}(\frac{G_{n}(x_{j})	}{S(I)})}dy_{j} d\PP_{n}(\lambda_{1},\ldots,\lambda_{n}) \\
& = \bigg( \frac{S(I)}{4} \bigg)^{k} \int_{I^{k}} \tilde{\phi}_{k,n}(y_{1},\ldots,y_{k}) dy_{1}\dots dy_{k},
\end{align*}
where $\PP_{n}$ refers to \eqref{def of eig distribution} and
\begin{align}\label{def of phit k n}
\tilde{\phi}_{k,n}(y_{1},\ldots,y_{k}) = \big( n\sqrt{2\log n} \big)^{k} \PP_{n}\Big( (y_{1},\ldots,y_{k}) \in \Sigma_{k}\big(\tfrac{G_{n}(x_{1})}{S(I)},\ldots,\tfrac{G_{n}(x_{k})}{S(I)}\big)\Big).
\end{align}
By \eqref{what we want to prove}, it remains to show that $\int_{I^{k}}\tilde{\phi}_{k,n} \to M(I)^{k} \prod_{j=1}^{k} e^{c_{0}-x_{j}}$ as $n\to + \infty$.

By \eqref{def of eig distribution} and Assumptions \ref{ass:V}, $\PP_{n}$ is absolutely continuous with respect to $d\lambda_{1}\ldots d\lambda_{n}$ on $\mathcal{S}^{k}$. In particular, $\int_{I^{k}}\phi_{k,n}=\int_{\overline{I}^{k}}\phi_{k,n}=\int_{(I^{\circ})^{k}}\phi_{k,n}$. Therefore, replacing $I$ by $\overline{I}$ if necessary, from now we assume without loss of generality that $I$ is a finite union of disjoint closed intervals, that is, $I = \cup_{j=1}^{p}[\mathfrak{a}_{j},\mathfrak{b}_{j}]$ with $p\in \N_{>0}$ and $\mathfrak{a}_{1}<\mathfrak{b}_{1}<\mathfrak{a}_{2}<\ldots<\mathfrak{b}_{p}$. %Moreover, since $\{(\lambda_{1},\ldots,\lambda_{n})\in \R^{n} :\lambda_{i}=\lambda_{j} \mbox{ for some } i \neq j\}$ has $\PP_{n}$-measure $0$, we can further assume that all $\lambda_{j}$'s are distinct. 
Consider 
\begin{align}\label{def of Lambda tilde}
\tilde{\Lambda}(I) = \{i : \lambda_{i},\lambda_{i+1} \mbox{ belong to the same connected component of } I\}.
\end{align}
Since $I=\overline{I}\subset \mathcal{S}^{\circ}$, the event $\Lambda(I) \neq \tilde{\Lambda}(I)$ is the event of a macroscopic gap within $\mathcal{S}$, so by standard large deviation principle (see e.g. \cite{AGZ2010}), one has
\begin{align*}
\PP_{n}(\Lambda(I) \neq \tilde{\Lambda}(I)) = \bigO(e^{-cn^{2}}), \qquad \mbox{as } n \to + \infty
\end{align*}
for some $c>0$. Hence, by Bayes' formula, 
\begin{align*}
\int_{I^{k}} \tilde{\phi}_{k,n}(y_{1},\ldots,y_{k}) dy_{1}\dots dy_{k} = \int_{I^{k}} \phi_{k,n}(y_{1},\ldots,y_{k}) dy_{1}\dots dy_{k}  + \bigO\bigg( \big( n\sqrt{\log n} \big)^{k}e^{-cn^{2}} \bigg)
\end{align*}
as $n\to + \infty$, where
\begin{align}
& \phi_{k,n}(y_{1},\ldots,y_{k}) = \big( n\sqrt{2\log n} \big)^{k}  \PP_{n}\Big( (y_{1},\ldots,y_{k}) \in \Sigma_{k}\big(\tfrac{G_{n}(x_{1})}{S(I)},\ldots,\tfrac{G_{n}(x_{k})}{S(I)}\big) , \; \Lambda(I) = \tilde{\Lambda}(I) \Big). \label{def of phi k n}
\end{align}
It remains to show that $\int_{I^{k}}\phi_{k,n} \to M(I)^{k} \prod_{j=1}^{k} e^{c_{0}-x_{j}}$ as $n\to + \infty$.

 As in \cite{FW2018}, we will prove this by establishing the following upper and lower bounds:
\begin{align}
\limsup_{n\to +\infty}  \int_{I^{k}} \phi_{k,n}(y_{1},\ldots,y_{k}) dy_{1}\dots dy_{k} \leq M(I)^{k} \prod_{j=1}^{k} e^{c_{0}-x_{j}}, \label{limsup} \\
\liminf_{n\to +\infty}  \int_{I^{k}} \phi_{k,n}(y_{1},\ldots,y_{k}) dy_{1}\dots dy_{k} \geq M(I)^{k} \prod_{j=1}^{k} e^{c_{0}-x_{j}}. \label{liminf}
\end{align}

\subsection{Upper bound \eqref{limsup}}
Let $x_{1},\ldots,x_{k}\in \R$ be fixed, and set
\begin{align}\label{def of An}
A_{n} & := \{(y_{1},\ldots,y_{k}) \in (I^{\circ})^{k} : [y_{i},y_{i}+\tfrac{G_{n}(x_{i})}{S(I)}] \cap [y_{j},y_{j}+\tfrac{G_{n}(x_{j})}{S(I)}] = \emptyset, \mbox{ for all } 1 \leq i < j \leq k\}.
\end{align}
By Lemma \ref{charac of Sigmak}, given $\lambda_{1}<\dots<\lambda_{n}$ satisfying $\Lambda(I)=\tilde{\Lambda}(I)$, one has
\begin{align*}
\Sigma_{k}\big(\tfrac{G_{n}(x_{1})}{S(I)},\ldots,\tfrac{G_{n}(x_{k})}{S(I)}\big) \subset A_{n}, \quad \mbox{ and thus } \quad ((I^{\circ})^{k}\setminus A_{n})\cap \Sigma_{k}\big(\tfrac{G_{n}(x_{1})}{S(I)},\ldots,\tfrac{G_{n}(x_{k})}{S(I)}\big) = \emptyset.
\end{align*}
This, combined with \eqref{def of phi k n}, implies that 
\begin{align}\label{lol34}
\phi_{k,n}(y_{1},\ldots,y_{k}) = 0 \quad \mbox{for } (y_{1},\ldots,y_{k})\in (I^{\circ})^{k}\setminus A_{n}.
\end{align}
Suppose now that $(y_{1},\ldots,y_{k}) \in A_{n}$. Denote by $\{\tilde{y}_{j}\}_{j=1}^{k+2p}$ the elements of the set $\{y_{j}\}_{j=1}^{k} \cup \{\mathfrak{a}_{j},\mathfrak{b}_{j}\}_{j=1}^{p}$ arranged so that $\tilde{y}_{i}<\tilde{y}_{j}$ if $i<j$. By Lemma \ref{charac of Sigmak}, one has
\begin{align}
\phi_{k,n}(y_{1},\ldots,y_{k}) = \big( n\sqrt{2\log n} \big)^{k}  \PP_{n}\Big( \xi^{(n)}(I_{n,k})=0, \; \xi^{(n)}(J_{n,k,j})>0, \nonumber \\
 \mbox{for all } 1 \leq j \leq k+2p-1 \mbox{ such that } \{\tilde{y}_{j},\tilde{y}_{j+1}\} \subsetneq \{\mathfrak{a}_{j},\mathfrak{b}_{j}\}_{j=1}^{p} \Big), \label{equality for phikn}
\end{align}
where
\begin{align}\label{def of Ink}
I_{n,k} := \cup_{j=1}^{k} \big[ y_{j},y_{j} + \tfrac{G_{n}(x_{j})}{S(I)} \big], \qquad J_{n,k,j} := [\tilde{y}_{j},\tilde{y}_{j+1}], \; 1 \leq j \leq k+2p-1.
\end{align}
In particular, $\phi_{k,n}(y_{1},\ldots,y_{k}) \leq \big( n\sqrt{2\log n} \big)^{k} \PP_{n}( \xi^{(n)}(I_{n,k})=0)$. Using Lemma \ref{lemma:negative correlation}, we then find
\begin{align}\label{lol35}
\phi_{k,n}(y_{1},\ldots,y_{k}) & \leq \big( n\sqrt{2\log n} \big)^{k} \prod_{j=1}^{k} \PP_{n}( \xi^{(n)}([y_{j},y_{j}+\tfrac{G_{n}(x_{j})}{S(I)}])=0 ).
\end{align}
By \eqref{lol34}, the inequality \eqref{lol35} also holds for $(y_{1},\ldots,y_{k})\in (I^{\circ})^{k}\setminus A_{n}$, and therefore
\begin{align*}
\int_{I^{k}} \phi_{k,n}(y_{1},\ldots,y_{k})dy_{1}\ldots dy_{k} \leq \big( n\sqrt{2\log n} \big)^{k}  \prod_{j=1}^{k} \int_{I} \PP_{n}( \xi^{(n)}([y,y+\tfrac{G_{n}(x_{j})}{S(I)}])=0 )dy.
\end{align*}
Hence, to prove \eqref{limsup}, it suffices to establish that
\begin{align}\label{lol31}
\limsup_{n\to + \infty} n\sqrt{2\log n} \int_{I} \PP_{n}( \xi^{(n)}([y,y+\tfrac{G_{n}(x)}{S(I)}])=0 )dy \leq M(I) e^{c_{0}-x}, \qquad x \in \R.
\end{align}
By Lemma \ref{lemma: comparison with CUE},
\begin{align*}
\PP_{n}( \xi^{(n)}([y,y+\tfrac{G_{n}(x)}{S(I)}])=0 ) = D_{n}\bigg(  \frac{2\pi \rho(y)}{S(I)}\frac{G_{n}(x)}{2} \bigg) + \bigO\bigg( \frac{1}{n(\log n)^{100}} \bigg), \qquad \mbox{as } n \to + \infty
\end{align*}
uniformly for $y \in I$, so \eqref{lol31} directly follows from Lemma \ref{lemma:big lemma}.

\subsection{Lower bound \eqref{liminf}}

Recall that $I$ consists of a finite union of intervals. Since $\rho$ is analytic and non-constant on $\mathcal{S}^{\circ}$, we can write $I = \cup_{j=1}^{m_{1}}I^{(j)} \cup \cup_{j=1}^{m_{2}} J^{(j)}$ for some $m_{1},m_{2}\in \N$ (as in the proof of Lemma \ref{lemma:big lemma}), where $\{I^{(j)},J^{(j)}\}$ are intervals on which $\rho$ is monotone and such that $\rho_{I} = \rho_{I^{(j)}}$ for each $j\in \{1,\ldots,m_{1}\}$ and $\rho_{I} < \rho_{J^{(j)}}$ for each $j\in \{1,\ldots,m_{2}\}$. For each $j \in \{1,m_{1}+m_{2}\}$, set
\begin{align*}
\rho^{-1}_{j} = \begin{cases}
(\rho|_{I^{(j)}})^{-1}, & \mbox{if } j \in \{1,\ldots,m_{1}\}, \\
(\rho|_{J^{(j-m_{1})}})^{-1}, & \mbox{if } j \in \{m_{1}+1,\ldots,m_{1}+m_{2}\}.
\end{cases}
\end{align*}
 For $j\in \{1,\ldots,m_{1}\}$, let us write $a^{(j)}$ and $b^{(j)}$ for the left and right endpoints of $I^{(j)}$, and let $q_{j}\in \{1,\ldots,q\}$ be the rate at which $\rho$ attains $\rho_{I}$ on $I^{(j)}$, i.e.
\begin{align}\label{lol14}
\rho(y) = \rho_{I} + \mathfrak{c}_{j}(y-a^{(j)})^{q_{j}} + \bigO((y-a^{(j)})^{q_{j}+1}), \qquad \mbox{as } y\to a^{(j)}
\end{align}
if $\rho$ is increasing on $I^{(j)}$, and
\begin{align}\label{lol15}
\rho(y) = \rho_{I} + \mathfrak{c}_{j}(b^{(j)}-y)^{q_{j}} + \bigO((b^{(j)}-y)^{q_{j}+1}), \qquad \mbox{as } y\to b^{(j)}
\end{align}
if $\rho$ is decreasing $I^{(j)}$. Let $\mathcal{A}_{*}:=\{j\in \{1,\ldots,m_{1}\}: q_{j}=q\}$. For each $j\in \mathcal{A}_{*}$, let $v^{(j)} = \arg\min \{\rho(a^{(j)}),\rho(b^{(j)})\} \in \{a^{(j)},b^{(j)}\}$.

\begin{lemma}\label{lemma:lower bound on phikn}
Fix $k\in \N_{>0}$, $\ell_{1},\ldots,\ell_{k}\in \mathcal{A}_{*}$, and $x_{1},\ldots,x_{k}\in \R$. Assume that all $u_{1},\ldots,u_{k} \in (0,+\infty)$ are all distinct. Then
\begin{align*}
\liminf_{n\to + \infty} \frac{1}{(2\log n)^{\frac{k}{q}}} \phi_{k,n}\Big(\rho^{-1}_{\ell_{1}}(\rho(v^{(\ell_{1})})[1+\tfrac{u_{1}}{\log n}]),\ldots,\rho^{-1}_{\ell_{k}}(\rho(v^{(\ell_{k})})[1+\tfrac{u_{k}}{\log n}])\Big) \geq \prod_{j=1}^{k}e^{c_{0}-x_{j}-2u_{j}}.
\end{align*}
\end{lemma}

We defer the proof of Lemma \ref{lemma:lower bound on phikn} to the end of this section.

Assuming Lemma \ref{lemma:lower bound on phikn}, we now prove the lower bound \eqref{liminf}. Set
\begin{align*}
I_{*} = \bigcup_{\substack{\ell=1 \\ q_{\ell}=q}}^{m_{1}} I^{(\ell)} = \bigcup_{\ell \in \mathcal{A}_{*}} I^{(\ell)}.
\end{align*}
Then
\begin{align}
&  \int_{I^{k}} \phi_{k,n}(y_{1},\ldots,y_{k}) dy_{1}\dots dy_{k} \geq   \int_{I_{*}^{k}} \phi_{k,n}(y_{1},\ldots,y_{k}) dy_{1}\dots dy_{k} \nonumber \\
& = \sum_{\ell_{1}\in \mathcal{A}_{*}} \dots \sum_{\ell_{k}\in \mathcal{A}_{*}} \int_{I^{(\ell_{1})}}\ldots \int_{I^{(\ell_{k})}} \phi_{k,n}(y_{1},\ldots,y_{k}) dy_{1}\dots dy_{k}. \label{lol36}
\end{align}
Fix $\ell_{1},\ldots,\ell_{k}\in \mathcal{A}_{*}$, and let $M>0$ be arbitrary but fixed. For each $j\in \{1,\ldots,k\}$, let
\begin{align*}
I^{(\ell_{j})}_{M} := \begin{cases}
\big(v^{(\ell_{j})},\rho_{\ell_{j}}^{-1}(\rho(v^{(\ell_{j})})(1+\frac{M}{\log n}))\big), & \mbox{if } v^{(\ell_{j})}=a^{(\ell_{j})}, \\
\big(\rho_{\ell_{j}}^{-1}(\rho(v^{(\ell_{j})})(1+\frac{M}{\log n})),v^{(\ell_{j})}\big), & \mbox{if } v^{(\ell_{j})}=b^{(\ell_{j})}.
\end{cases}
\end{align*}
For all large enough $n$, $I^{(\ell_{j})}_{M}$ is well-defined and satisfies $I^{(\ell_{j})}_{M}\subset I^{(\ell_{k})}$. Using the change of variables $\rho(y_{j}) = t_{j}$ and then $\frac{t_{j}}{\rho_{I}}=1+\frac{u_{j}}{\log n}$, $j=1,\ldots,k$, we get, for all large enough $n$,
\begin{align*}
& \int_{I^{(\ell_{1})}}\ldots \int_{I^{(\ell_{k})}} \phi_{k,n}(y_{1},\ldots,y_{k}) dy_{1}\dots dy_{k} \geq \int_{I_{M}^{(\ell_{1})}}\ldots \int_{I_{M}^{(\ell_{k})}} \phi_{k,n}(y_{1},\ldots,y_{k}) dy_{1}\dots dy_{k} \\
& = \int_{\rho(v^{(\ell_{1})})}^{\rho(v^{(\ell_{1})})(1+\frac{M}{\log n})} \ldots \int_{\rho(v^{(\ell_{k})})}^{\rho(v^{(\ell_{k})})(1+\frac{M}{\log n})} \phi_{k,n}(\rho_{\ell_{1}}^{-1}(t_{1}),\ldots,\rho_{\ell_{k}}^{-1}(t_{k})) \prod_{j=1}^{k}\big|(\rho_{\ell_{j}}^{-1})'(t_{j})\big|dt_{j} \\
& = \frac{\rho_{I}^{k}}{(\log n)^{k}} \int_{[0,M]^{k}} \phi_{k,n}\Big(\rho_{\ell_{1}}^{-1}(\rho_{I}[1+\tfrac{u_{1}}{\log n}]),\ldots,\rho_{\ell_{k}}^{-1}(\rho_{I}[1+\tfrac{u_{k}}{\log n}])\Big) \prod_{j=1}^{k}\big|(\rho_{\ell_{j}}^{-1})'(\rho_{I}[1+\tfrac{u_{j}}{\log n}])\big|du_{j}.
\end{align*}
By similar computations as in \eqref{lol7} and \eqref{lol7 bis}, for any $j\in \{1,\ldots,k\}$, we obtain
\begin{align}\label{lol13}
\big|(\rho_{\ell_{j}}^{-1})'(\rho_{I}(1+\tfrac{u_{j}}{\log n}))\big| = \frac{1}{d_{v^{(\ell_{j})}}^{1/q}q} \bigg( \frac{\rho_{I}u_{j}}{\log n} \bigg)^{\frac{1}{q}-1} \bigg( 1 + \bigO \Big( \frac{u_{j}^{1/q}}{(\log n)^{1/q}} \Big) \bigg), \qquad \mbox{as } n\to +\infty
\end{align}
uniformly for $u_{j}\in [0,M]$, where $d_{u}$ is as in \eqref{def of du}. Using first Fatou's Lemma and then Lemma \ref{lemma:lower bound on phikn} and \eqref{lol13}, we find
\begin{align*}
& \liminf_{n\to +\infty}  \int_{I^{(\ell_{1})}}\ldots \int_{I^{(\ell_{k})}} \phi_{k,n}(y_{1},\ldots,y_{k}) dy_{1}\dots dy_{k} \\
& \geq  \int_{[0,M]^{k}} \liminf_{n\to +\infty}  \frac{\rho_{I}^{k}}{(\log n)^{k}}\phi_{k,n}\Big(\rho_{\ell_{1}}^{-1}(\rho_{I}[1+\tfrac{u_{1}}{\log n}]),\ldots,\rho_{\ell_{k}}^{-1}(\rho_{I}[1+\tfrac{u_{k}}{\log n}])\Big)  \prod_{j=1}^{k}\big|(\rho_{\ell_{j}}^{-1})'(\rho_{I}[1+\tfrac{u_{j}}{\log n}])\big|du_{j} \\
& \geq \int_{[0,M]^{k}} \liminf_{n\to +\infty}  \frac{\rho_{I}^{k}}{(\log n)^{k}}\phi_{k,n}\Big(\rho^{-1}(\rho(a)[1+\tfrac{u_{1}}{\log n}]),\ldots,\rho^{-1}(\rho(a)[1+\tfrac{u_{k}}{\log n}])\Big) \prod_{j=1}^{k} \bigg( \frac{\rho_{I}u_{j}}{\log n} \bigg)^{\frac{1}{q}-1} \frac{du_{j}}{d_{v^{(\ell_{j})}}^{1/q}q}  \\
& \geq  \prod_{j=1}^{k} \frac{2^{1/q}\rho_{I}^{1/q}}{d_{v^{(\ell_{j})}}^{1/q}q}e^{c_{0}-x_{j}} \int_{0}^{M}u^{\frac{1}{q}-1}e^{-2u}du.
\end{align*}
Taking $M\to + \infty$ in the above yields
\begin{align}\label{lol37}
\liminf_{n\to +\infty}  \int_{I^{(\ell_{1})}}\ldots \int_{I^{(\ell_{k})}} \phi_{k,n}(y_{1},\ldots,y_{k}) dy_{1}\dots dy_{k} \geq \prod_{j=1}^{k} \frac{\rho_{I}^{1/q}}{d_{v^{(\ell_{j})}}^{1/q}q}e^{c_{0}-x_{j}} \Gamma(\tfrac{1}{q}).
\end{align}
Taking the $\liminf$ as $n\to + \infty$ in \eqref{lol36}, and then using \eqref{lol37}, we find
\begin{align*}
& \liminf_{n\to +\infty}  \int_{I^{k}} \phi_{k,n}(y_{1},\ldots,y_{k}) dy_{1}\dots dy_{k} \geq \sum_{\ell_{1}\in \mathcal{A}_{*}} \dots \sum_{\ell_{k}\in \mathcal{A}_{*}} \prod_{j=1}^{k} \frac{\rho_{I}^{1/q}}{d_{v^{(\ell_{j})}}^{1/q}q}e^{c_{0}-x_{j}} \Gamma(\tfrac{1}{q}) \\
& = \bigg( \sum_{u\in \mathcal{A}}  d_{u}^{-\frac{1}{q}} + 2 \sum_{u\in \mathcal{B}}  d_{u}^{-\frac{1}{q}} \bigg)^{k} \prod_{j=1}^{k} \frac{\rho_{I}^{1/q}}{q}e^{c_{0}-x_{j}} \Gamma(\tfrac{1}{q}) = M(I)^{k} \prod_{j=1}^{k}e^{c_{0}-x_{j}},
\end{align*}
which is \eqref{liminf}.

\medskip It remains to prove Lemma \ref{lemma:lower bound on phikn}. We first establish a lower bound for $\PP_{n}(\xi^{(n)}(J)=0)$ when $J$ is a finite union of disjoint intervals of size $\sqrt{\log n}/n$. %The proof of Lemma \ref{lemma:lower bound for the decoupling} closely follows that of \cite[Lemma~14]{FW2018}; however, since it relies on the new Lemma~\ref{lemma:kernel}, we provide the details for completeness.

\begin{lemma}\label{lemma:lower bound for the decoupling}
Let $\epsilon_{0} > 0$, $k\in \N_{>0}$, and let $I$ be a compact subset of $\mathcal{S}^{\circ}$. Assume $y_{1},\ldots,y_{k}\in I^{\circ}$, 
\begin{align}\label{orders of the aj}
a_{j} \asymp \tfrac{\sqrt{\log n}}{n}, \qquad \mbox{as } n \to + \infty, \quad \mbox{for all } j\in \{1,\ldots,k\},
\end{align}
and that, for all large $n$, 
\begin{align}
& |y_{i}-y_{j}| \geq \tfrac{\epsilon_{0}}{(\log n)^{1/q}} & & \mbox{for every } 1 \leq i < j \leq k, \label{cond1} \\
& \mathrm{dist}(y_{j},\partial I) \geq \tfrac{\epsilon_{0}}{(\log n)^{1/q}}, & & \mbox{for every } 1 \leq j \leq k. \label{cond2}
%& \tfrac{\rho(y_{j})}{\rho_{I}} \leq 1+\tfrac{C_{0}}{\log n}, & & \mbox{for every } 1 \leq i \leq k. \label{cond3}
\end{align}
Then, for all large enough $n$, we have
\begin{align}\label{lol17}
\PP_{n} \big( \xi^{(n)}(\cup_{j=1}^{k}[y_{j},y_{j}+a_{j}]) = 0 \big) \geq \big( 1-(\log n)^{-100} \big) \prod_{j=1}^{k} D_{n}(\pi a_{j} \rho(y_{j})).
\end{align}
\end{lemma}
\begin{proof}
The proof is adapted from \cite[proof of Lemma~14]{FW2018} and relies on Lemma \ref{lemma:kernel}. 

The assumptions \eqref{orders of the aj}--\eqref{cond2} imply that $J:=\cup_{j=1}^{k}[y_{j},y_{j}+a_{j}]$ is a disjoint union satisfying $J \subset I$ for all sufficiently large $n$. The probability on the left-hand side of \eqref{lol17} can be written as the following Fredholm determinant:
\begin{align*}
\PP_{n} \big( \xi^{(n)}(\cup_{j=1}^{k}[y_{j},y_{j}+a_{j}]) = 0 \big) = \PP_{n}\big( \xi^{(n)}(J)=0 \big) = \det(I-A),
\end{align*}
where $A = \chi_{J} \mathcal{K}_{n} \chi_{J}$ and $\mathcal{K}_{n}$ is the integral operator whose kernel $K_{n}$ is given by \eqref{def of Kn}. For $1 \leq i,j \leq k$, we define
\begin{align}\label{def of Aij}
A_{i,j} = \chi_{[y_{i},y_{i}+a_{i}]} \mathcal{K}_{n} \chi_{[y_{j},y_{j}+a_{j}]}, \quad \mbox{ so that } \quad A = \sum_{i=1}^{k}\sum_{j=1}^{k}A_{i,j}.
\end{align}
Let $B_{j}$ the integral operator with kernel
\begin{align}\label{def of Bj}
B_{j}(u,v) = 2\pi \rho(y_{j}) K_{n,(y_{j},y_{j}+a_{j})}^{\mathrm{CUE}}(2\pi\rho(y_{j})u,2\pi \rho(y_{j})v),
\end{align}
where $K_{n,(y_{j},y_{j}+a_{j})}^{\mathrm{CUE}}(x_{1},x_{2}) := \chi_{(y_{j},y_{j}+a_{j})}(x_{1})K_{n}^{\mathrm{CUE}}(x_{1},x_{2})\chi_{(y_{j},y_{j}+a_{j})}(x_{2})$. We also write $B=\sum_{j=1}^{k}B_{j}$.

\smallskip For all $j\in \{1,\ldots,k\}$ and all large enough $n$, one has $\rho(y_{j})a_{j}<1$, and thus
\begin{align*}
\det(I-B_{j}) = \PP_{n}^{\mathrm{CUE}}(\theta_{i}\notin [0,2\pi \rho(y_{j})a_{j}], \; 1 \leq i \leq n) = D_{n}(\pi \rho(y_{j})a_{j}).
\end{align*}
Moreover, for all sufficiently large $n$, the intervals $(y_{j},y_{j}+a_{j})$ are disjoint, which implies $B_{i}B_{j}=0$ for all $1 \leq i \neq j \leq k$, and thus
\begin{align*}
\det(I-B) = \prod_{j=1}^{k}\det(I-B_{j}) = \prod_{j=1}^{k} D_{n}(\pi \rho(y_{j})a_{j}).
\end{align*}
We will compare the two Fredholm determinants using \cite[Lemma 6]{FW2018}, namely
\begin{align}\label{lol38}
\frac{\det(I-A)}{\det(I-B)} \geq e^{-b_{2}}(1-b_{3}),
\end{align}
where
\begin{align}\label{def of b2 and b3}
b_{2} = \tr (A-B)(I-B)^{-1}, \qquad b_{3} = |B-A|_{2}^{2} \| (I-B)^{-1} \|^{2}.
\end{align}
Recall from \eqref{upper bound Tr resolvant} (with $A,B$ replaced by $-A,-B$) that
\begin{align}
& |b_{2}| \leq |\tr(A-B)| + |A-B|_{2}|B|_{2} \| (I-B)^{-1} \|. \label{upper bound Tr resolvant bis}
\end{align}
The following estimates were established in \cite[Proof of Lemma 14]{FW2018}:
\begin{align*}
& |B|_{2}^{2} = \bigO\big( \sqrt{\log n} \big), \qquad \|(I-B)^{-1}\| = \bigO \bigg( n(\log n)^{-\frac{1}{2}}e^{-\frac{1}{2}\sqrt{\log n}} \bigg), \qquad \mbox{as } n \to + \infty.
\end{align*}
It remains to estimate $|A-B|_{2}$ and $\tr (A-B)$. Using \eqref{def of Aij} and \eqref{def of Bj}, we obtain
\begin{align}\label{lol16}
|A-B|_{2}^{2} = \sum_{j=1}^{k}|A_{j,j}-B_{j}|_{2}^{2}+\sum_{i\neq j}|A_{i,j}|_{2}^{2}, \qquad \tr(A-B) = \sum_{j=1}^{k}\tr(A_{j,j}-B_{j}).
\end{align}
It follows from \eqref{cond1} and \eqref{orders of the aj} that for $x\in [y_{i},y_{i}+a_{i}] \subset I$ and $y\in [y_{j},y_{j}+a_{j}] \subset I$ with $i\neq j$, $1 \leq i,j \leq k$, we have $|x-y| \geq \frac{\epsilon_{0}}{2(\log n)^{1/q}}$ for all sufficiently large $n$. Hence, by \eqref{rough estimate on Kn},
\begin{align*}
K_{n}(x,y) = \bigO((\log n)^{1/q}), \qquad \mbox{as } n \to + \infty
\end{align*}
uniformly for $x\in [y_{i},y_{i}+a_{i}]$ and $y\in [y_{j},y_{j}+a_{j}]$ with $i \neq j$. Using again \eqref{orders of the aj}, we obtain for $i \neq j$ 
\begin{align}\label{Aij norm}
|A_{i,j}|_{2}^{2} = \int_{y_{i}}^{y_{i}+a_{i}}\int_{y_{j}}^{y_{j}+a_{j}} |K_{n}(x,y)|^{2}dxdy = \bigO(a_{i}a_{j}(\log n)^{2/q}) = \bigO \bigg( \frac{(\log n)^{1+\frac{2}{q}}}{n^{2}} \bigg)
\end{align}
as $n\to + \infty$. We now estimate the other terms in \eqref{lol16}. By \eqref{asymp Kn general V} and \cite[Lemma 3.4]{BAB2013}, we have
\begin{align*}
& \frac{1}{n\rho(y_{j})}K_{n}(x,y) = \frac{\sin\big( \pi n\rho(y_{j})(x - y) \big)}{\pi n\rho(y_{j})(x - y)} + \bigO\bigg( \frac{1}{n}+a_{j} \bigg), \\
& \frac{2\pi}{n}K_{n}^{\mathrm{CUE}}(2\pi \rho(y_{j})x,2\pi \rho(y_{j})y) = \frac{\sin\big( \pi n \rho(y_{j})(x - y) \big)}{\pi n\rho(y_{j}) (x-y)} + \bigO\bigg(\frac{a_{j}}{n}\bigg),
\end{align*}
as $n\to + \infty$ uniformly for $x,y\in [y_{j},y_{j}+a_{j}]$. Thus, as $n\to + \infty$,
\begin{align}
& |A_{j,j}-B_{j}|_{2}^{2}  = \int_{y_{j}}^{y_{j}+a_{j}}\int_{y_{j}}^{y_{j}+a_{j}} \bigO\big((1+n a_{j})^{2}\big)dxdy = \bigO\big(a_{j}^{2}(1+n a_{j})^{2}\big) = \bigO\bigg( \frac{(\log n)^{2}}{n^{2}} \bigg), \label{Ajj-Bj norm} \\
& \tr (A_{j,j}-B_{j}) = \int_{y_{j}}^{y_{j}+a_{j}} \bigO\big(1+n a_{j}\big)dx = \bigO\big(a_{j}(1+n a_{j})\big) = \bigO \bigg( \frac{\log n}{n} \bigg). \label{Ajj-Bj trace}
\end{align}
Substituting \eqref{Aij norm}, \eqref{Ajj-Bj norm} and \eqref{Ajj-Bj trace} in \eqref{lol16} yields
\begin{align*}
|A-B|_{2}^{2} = \bigO \bigg( \frac{(\log n)^{1+\frac{2}{q}}}{n^{2}} \bigg), \qquad \tr(A-B) = \bigO \bigg( \frac{\log n}{n} \bigg), \qquad \mbox{as } n \to + \infty.
\end{align*}
Hence, by \eqref{upper bound Tr resolvant bis},
\begin{align*}
|b_{2}| & \leq |\tr(A-B)| + |A-B|_{2}|B|_{2} \| (I+B)^{-1} \| \\
& = \bigO \bigg( \frac{\log n}{n} + \frac{(\log n)^{\frac{1}{2}+\frac{1}{q}+\frac{1}{4}}}{n} \times n(\log n)^{-\frac{1}{2}}e^{-\frac{1}{2}\sqrt{\log n}} \bigg) = \bigO\big( (\log n)^{-101} \big), \qquad \mbox{as } n \to + \infty,
\end{align*}
and, by \eqref{def of b2 and b3},
\begin{align*}
b_{3} = \bigO \bigg( \frac{(\log n)^{1+\frac{2}{q}}}{n^{2}} \times n^{2}(\log n)^{-1}e^{-\sqrt{\log n}} \bigg) = \bigO \big( (\log n)^{-101} \big), \quad \mbox{as } n \to + \infty.
\end{align*}
It now directly follows from \eqref{lol38} that
\begin{align*}
\frac{\det(I-A)}{\det(I-B)} \geq e^{-b_{2}}(1-b_{3}) \geq 1- C (\log n)^{-101},
\end{align*}
holds for some $C>0$ and all sufficiently large $n$, which implies \eqref{lol17}.
\end{proof}
Now we prove Lemma \ref{lemma:lower bound on phikn}. Fix $k\in \N_{>0}$, $\ell_{1},\ldots,\ell_{k}\in \mathcal{A}_{*}$, $x_{1},\ldots,x_{k}\in \R$, and $u_{1},\ldots,u_{k} \in (0,+\infty)$, such that $u_{i}\neq u_{j}$ for $i\neq j$. For $j=1,\ldots,k$, set
\begin{align}\label{lol39}
y_{j} := \rho_{\ell_{j}}^{-1}(\rho(v^{(\ell_{j})})[1+\tfrac{u_{j}}{\log n}]), \qquad a_{j} := \frac{G_{n}(x_{j})}{S(I)}.
\end{align}
Then $y_{j} \in I^{\circ}$ for all sufficiently large $n$ (because $u_{j}\neq 0$). Since $u_{i}\neq u_{j} \neq 0$ for $1 \leq i \neq j \leq k$, by \eqref{expansion of rho inv}, we have $|y_{i}-y_{j}| \geq \epsilon_{0} (\log n)^{-1/q}$ and $\mathrm{dist}(y_{j},\partial I) \geq \epsilon_{0} (\log n)^{-1/q}$ for some $\epsilon_{0}>0$. It also directly follows from \eqref{def of Gn} that $a_{j} \asymp \sqrt{\log n}/n$ as $n\to + \infty$. This implies that the intervals $[y_{i},y_{i}+a_{i}]=[y_{i},y_{i}+\frac{G_{n}(x_{j})}{S(I)}]$ are disjoint for all large enough $n$. Thus $(y_{1},\ldots,y_{k})\in A_{n}$ (recall \eqref{def of An}) for all large enough $n$, and, by \eqref{equality for phikn},
\begin{align}
& \frac{1}{(2\log n)^{\frac{k}{q}}}\phi_{k,n}(y_{1},\ldots,y_{k}) = n^{k}( 2\log n)^{\frac{k}{2}-\frac{k}{q}} \PP_{n}\Big( \xi^{(n)}(I_{n,k})=0, \; \xi^{(n)}(J_{n,k,j})>0, \nonumber \\
& \qquad \hspace{3.5cm} \mbox{ for all } 1 \leq j \leq k+2p-1 \mbox{ such that } \{\tilde{y}_{j},\tilde{y}_{j+1}\} \subsetneq \{\mathfrak{a}_{j},\mathfrak{b}_{j}\}_{j=1}^{p} \Big) \nonumber \\
& \qquad \geq n^{k}( 2\log n)^{\frac{k}{2}-\frac{k}{q}} \PP_{n}\big( \xi^{(n)}(I_{n,k})=0 \big) \nonumber \\
& \qquad - n^{k}( 2\log n)^{\frac{k}{2}-\frac{k}{q}} \sum_{j \in \mathcal{C}} \PP_{n}\Big( \xi^{(n)}(I_{n,k})=0= \xi^{(n)}(J_{n,k,j}) \Big), \label{lol19}
\end{align}
where $\mathcal{C} := \{j \in \{1,\ldots,k+2p-1\}: \{\tilde{y}_{j},\tilde{y}_{j+1}\} \subsetneq \{\mathfrak{a}_{j},\mathfrak{b}_{j}\}_{j=1}^{p}\}$.

\begin{lemma}\label{lemma:some small prob}
For each $j\in \mathcal{C}$, we have
\begin{align}\label{lol18}
n^{k}( 2\log n)^{\frac{k}{2}-\frac{k}{q}}\PP_{n}\Big( \xi^{(n)}(I_{n,k})=0= \xi^{(n)}(J_{n,k,j}) \Big) = o(1), \qquad \mbox{as } n \to + \infty.
\end{align}
\end{lemma}
\begin{proof}
Recall from \eqref{def of Ink} that $I_{n,k} = \cup_{l=1}^{k}[y_{l},y_{l}+a_{l}]$ and $J_{n,k,j} = [\tilde{y}_{j},\tilde{y}_{j+1}]$, where $\{\tilde{y}_{l}\}_{l=1}^{k+2p}$ are the elements of the set $\{y_{l}\}_{l=1}^{k} \cup \{\mathfrak{a}_{l},\mathfrak{b}_{l}\}_{l=1}^{p}$ arranged so that $\tilde{y}_{i}<\tilde{y}_{l}$ if $i<l$. Set $\tilde{y}_{j}':=(\tilde{y}_{j}+\tilde{y}_{j+1})/2$ and $d_{0}:=\frac{G_{n}(-M)}{S(I)}$, where $M:=\max\{|x_{1}|,\ldots,|x_{k}|\}+1$. Note that $\tilde{y}_{j}'\in I^{\circ}$ because $j\in \mathcal{C}$. For all sufficiently large $n$, one has $d_{0} \leq a_{l}$ for all $l=1,\ldots,k$, and
\begin{align}
0 & \leq \PP_{n}\Big( \xi^{(n)}(I_{n,k})=0= \xi^{(n)}(J_{n,k,j}) \Big) = \PP_{n}\Big( \xi^{(n)}(I_{n,k}\cup J_{n,k,j})=0 \Big) \nonumber \\
& \leq \PP_{n}\Big( \xi^{(n)}(\cup_{l=1}^{k}[y_{l},y_{l}+d_{0}]\cup [\tilde{y}_{j}',\tilde{y}_{j}'+d_{0}])=0 \Big) \nonumber \\
& \leq \PP_{n}\big( \xi^{(n)}([\tilde{y}_{j}',\tilde{y}_{j}'+d_{0}])=0 \big) \prod_{l=1}^{k} \PP\big( \xi^{(n)}([y_{l},y_{l}+d_{0}])=0 \big), \label{lol40}
\end{align}
where we used Lemma \ref{lemma:negative correlation} in the last step. Using now Lemmas \ref{lemma: comparison with CUE} and \ref{lemma:lim of Dn}, we find
\begin{align}
& \PP_{n}\big( \xi^{(n)}([\tilde{y}_{j}',\tilde{y}_{j}'+d_{0}])=0 \big) \prod_{l=1}^{k} \PP\big( \xi^{(n)}([y_{l},y_{l}+d_{0}])=0 \big) = D_{n}(\pi \rho(\tilde{y}_{j}')d_{0}) \prod_{l=1}^{k}D_{n}(\pi \rho(y_{l})d_{0})(1+o(1)) \nonumber \\
& \leq \big(D_{n}(S(I)d_{0}/2)\big)^{k+1}(1+o(1)) = \big(D_{n}(\tfrac{G_{n}(-M)}{2})\big)^{k+1}(1+o(1)) \nonumber \\
& = \bigg(\frac{(2\log n)^{\frac{1}{q}-\frac{1}{2}}}{n}e^{c_{0}+M}\bigg)^{k+1}(1+o(1)), \qquad \mbox{as } n \to + \infty. \label{lol41}
\end{align}
The asymptotic formula \eqref{lol18} now directly follows from \eqref{lol40} and \eqref{lol41}.
\end{proof}
By \eqref{lol19} and Lemma \ref{lemma:some small prob},
\begin{align*}
\frac{1}{(2\log n)^{\frac{k}{q}}}\phi_{k,n}(y_{1},\ldots,y_{k}) & \geq n^{k}( 2\log n)^{\frac{k}{2}-\frac{k}{q}} \PP\big( \xi^{(n)}(I_{n,k})=0 \big) -o(1), \qquad \mbox{as } n \to + \infty.
\end{align*}
Using now Lemma \ref{lemma:lower bound for the decoupling}, and then Lemma \ref{lemma:lim of Dn}, we get
\begin{align*}
\frac{1}{(2\log n)^{\frac{k}{q}}}\phi_{k,n}(y_{1},\ldots,y_{k}) & \geq \big( 1-o(1) \big) \prod_{j=1}^{k} n( 2\log n)^{\frac{1}{2}-\frac{1}{q}} D_{n}(\pi a_{j} \rho(y_{j})) -o(1) \\
& \geq \big( 1-o(1) \big) \prod_{j=1}^{k} n( 2\log n)^{\frac{1}{2}-\frac{1}{q}} D_{n}\Big( \big( 1+\tfrac{u_{j}}{\log n}  \big)\frac{G_{n}(x_{j})}{2}\Big) -o(1) \\
& = \big( 1-o(1) \big) \prod_{j=1}^{k} e^{c_{0}-x_{j}-2u_{j}}  -o(1),
\end{align*}
which finishes the proof of Lemma \ref{lemma:lower bound on phikn}.

\appendix

\section{A characterization of $\Sigma_{k}(a_{1},\ldots,a_{k})$}
The following lemma is a generalization of \cite[Lemma 11]{FW2018} to the case where $I$ is a finite union of closed intervals. Recall that $\Lambda(I) = \{i : \lambda_{i},\lambda_{i+1}\in I\}$ and that $\tilde{\Lambda}(I)$ is defined in \eqref{def of Lambda tilde}. Recall also that $\Sigma_{k}(a_{1},\ldots,a_{k})$ is defined in \eqref{def of Sigmak}.
\begin{lemma}\label{charac of Sigmak}
Suppose $I$ is a finite union of closed intervals, that is, $I = \cup_{j=1}^{p}[\mathfrak{a}_{j},\mathfrak{b}_{j}]$ with $p\in \N_{>0}$ and $\mathfrak{a}_{1}<\mathfrak{b}_{1}<\mathfrak{a}_{2}<\ldots<\mathfrak{b}_{p}$. %Let $y_{1},\ldots,y_{k}\in I^{\circ}$ be all distinct, and denote by $\{\tilde{y}_{j}\}_{j=1}^{k+2p}$ the elements of the set $\{y_{j}\}_{j=1}^{k} \cup \{\mathfrak{a}_{j},\mathfrak{b}_{j}\}_{j=1}^{p}$ arranged so that $\tilde{y}_{i}< \tilde{y}_{j}$ if $i<j$.

Given $\lambda_{1}<\dots<\lambda_{n}$ satisfying $\Lambda(I)=\tilde{\Lambda}(I)$, the condition $(y_{1},\ldots,y_{k})\in \Sigma_{k}(a_{1},\ldots,a_{k})$ is equivalent to the following conditions:
\begin{itemize}
\item[(i)] $y_{1},\ldots,y_{k}\in I^{\circ}$ and $[y_{l},y_{l}+a_{l}] \cap [y_{j},y_{j}+a_{j}] = \emptyset$ for $1 \leq l < j \leq k$,
\item[(ii)] $\lambda_{l} \notin [y_{j},y_{j}+a_{j}]$ for $1 \leq j \leq k$, $1 \leq l \leq n$,
\item[(iii)] $\{\lambda_{1},\ldots,\lambda_{n}\}\cap [\tilde{y}_{j},\tilde{y}_{j+1}] \neq \emptyset$ for every $j \in \{1,\ldots,k+2p-1\}$ such that $\{\tilde{y}_{j},\tilde{y}_{j+1}\} \subsetneq \{\mathfrak{a}_{j},\mathfrak{b}_{j}\}_{j=1}^{p}$, where $\{\tilde{y}_{j}\}_{j=1}^{k+2p}$ are the elements of $\{y_{j}\}_{j=1}^{k} \cup \{\mathfrak{a}_{j},\mathfrak{b}_{j}\}_{j=1}^{p}$ arranged so that $\tilde{y}_{i}< \tilde{y}_{j}$ if $i<j$.
%\item[(iv)] $y_{1},\ldots,y_{k}\in I$.
\end{itemize}
\end{lemma}
\begin{proof}
By \eqref{def of Sigmak}, if $(y_{1},\ldots,y_{k})\in \Sigma_{k}(a_{1},\ldots,a_{k})$, then $y_{1},\ldots,y_{k}$ must be all distinct. Conversely, if $y_{1},\ldots,y_{k}$ satisfy (i), then they are necessarily all distinct as well. Denote by $\{\hat{y}_{j}\}_{j=1}^{k}$ the elements of the set $\{y_{j}\}_{j=1}^{k}$ arranged so that $\hat{y}_{i}<\hat{y}_{j}$ if $i<j$. Thus $\hat{y}_{j}=y_{\sigma(j)}$ for some permutation $\sigma$ of $\{1,\ldots,k\}$. For convenience, set $\hat{a}_{j} := a_{\sigma(j)}$. Note from \eqref{def of Sigmak} that $(y_{1},\ldots,y_{k})\in \Sigma_{k}(a_{1},\ldots,a_{k})$ if and only if $(\hat{y}_{1},\ldots,\hat{y}_{k})\in \Sigma_{k}(\hat{a}_{1},\ldots,\hat{a}_{k})$.

If $(y_{1},\ldots,y_{k})\in \Sigma_{k}(a_{1},\ldots,a_{k})$, then by \eqref{def of Sigmak} we can find $i_{1},\ldots,i_{k}\in \Lambda(I)$ all distinct such that
\begin{align}\label{lol33}
\hat{y}_{j} \in J_{i_{j}}(\hat{a}_{j}) = (\lambda_{i_{j}},\lambda_{i_{j}+1}-\hat{a}_{j})_{+} \subset I^{\circ}, \qquad j=1,\ldots,k.
\end{align}
This already implies $y_{1},\ldots,y_{k}\in I^{\circ}$. Since $\hat{y}_{l}<\hat{y}_{j}$ for $l<j$, we must have $i_{l}<i_{j}$ for $l<j$. Let $m_{1},\ldots,m_{p}$ be integers satisfying $m_{1}+\ldots+m_{p}=k$ and such that 
\begin{align}\label{lol32}
\hat{y}_{l} \in (\mathfrak{a}_{j},\mathfrak{b}_{j}), \qquad 1 \leq j \leq p, \; 1+k_{j-1} \leq l \leq k_{j}, \qquad k_{j} := \sum_{i=1}^{j} m_{i}.
\end{align}
Since $i_{k_{1}} \in \Lambda(I) = \tilde{\Lambda}(I)$, we have $\lambda_{i_{k_{1}}+1} \leq \mathfrak{b}_{1}$, and thus
\begin{align*}
& \mathfrak{a}_{1} \leq \lambda_{i_{1}} < \hat{y}_{1} < \lambda_{i_{1}+1}-\hat{a}_{1} < \lambda_{i_{1}+1} \leq \lambda_{i_{2}} < \hat{y}_{2} < \lambda_{i_{2}+1}-\hat{a}_{2} < \lambda_{i_{2}+1} \leq \lambda_{i_{3}} < \hat{y}_{3} < \ldots \\
&  \leq \lambda_{i_{k_{1}-1}} < \hat{y}_{k_{1}-1} < \lambda_{i_{k_{1}-1}+1}-\hat{a}_{k_{1}-1} < \lambda_{i_{k_{1}-1}+1} \leq \lambda_{i_{k_{1}}} < \hat{y}_{k_{1}} < \lambda_{i_{k_{1}}+1}-\hat{a}_{k_{1}} < \lambda_{i_{k_{1}}+1} \leq \mathfrak{b}_{1}.
\end{align*}
More generally, for each $j\in \{1,\ldots,p\}$,
\begin{align*}
& \mathfrak{a}_{j} \leq \lambda_{i_{k_{j-1}+1}} < \hat{y}_{k_{j-1}+1} < \lambda_{i_{k_{j-1}+1}+1}-\hat{a}_{k_{j-1}+1} < \lambda_{i_{k_{j-1}+1}+1} \leq \lambda_{i_{k_{j-1}+2}} < \hat{y}_{k_{j-1}+2} < \ldots \\
&  \leq \lambda_{i_{k_{j}-1}} < \hat{y}_{k_{j}-1} < \lambda_{i_{k_{j}-1}+1}-\hat{a}_{k_{j}-1} < \lambda_{i_{k_{j}-1}+1} \leq \lambda_{i_{k_{j}}} < \hat{y}_{k_{j}} < \lambda_{i_{k_{j}}+1}-\hat{a}_{k_{j}} < \lambda_{i_{k_{j}}+1} \leq \mathfrak{b}_{j},
\end{align*}
which implies $(i)$, $(ii)$ and $(iii)$. 

Conversely, if $y_{1},\ldots,y_{k}\in I^{\circ}$, then there exist integers $m_{1},\ldots,m_{p}$ satisfying $m_{1}+\ldots+m_{p}=k$ and such that \eqref{lol32} holds. If $(y_{1},\ldots,y_{k})$ also satisfies $(i)$, $(ii)$, $(iii)$, then one can find $1 \leq i_{1} < \ldots < i_{k_{1}} \leq n-1$ such that
\begin{align*}
& \mathfrak{a}_{1} \leq \lambda_{i_{1}} < \hat{y}_{1} < \hat{y}_{1}+\hat{a}_{1} < \lambda_{i_{1}+1} \leq \lambda_{i_{2}} < \hat{y}_{2} < \hat{y}_{2}+\hat{a}_{2} < \ldots \\
& \leq \lambda_{i_{k_{1}-1}} < \hat{y}_{k_{1}-1} < \hat{y}_{k_{1}-1}+\hat{a}_{k_{1}-1} < \lambda_{i_{k_{1}-1}+1} \leq \lambda_{i_{k_{1}}} < \hat{y}_{k_{1}} < \hat{y}_{k_{1}}+\hat{a}_{k_{1}} < \lambda_{i_{k_{1}+1}} \leq \mathfrak{b}_{1}.
\end{align*}
This implies $i_{1},\ldots,i_{k_{1}}\in \Lambda(I)$ and that \eqref{lol33} holds with $j=1,\ldots,k_{1}$. The above argument similarly extends to the other components of $I$: for any $j\in \{1,\ldots,p\}$, there exists $1 \leq i_{k_{j-1}+1} < \ldots < i_{k_{j}} \leq n-1$ such that
\begin{align*}
& \mathfrak{a}_{j} \leq \lambda_{i_{k_{j-1}+1}} < \hat{y}_{k_{j-1}+1} < \hat{y}_{k_{j-1}+1}+\hat{a}_{k_{j-1}+1} < \lambda_{i_{k_{j-1}+1}+1} \leq \lambda_{i_{k_{j-1}+2}} < \hat{y}_{k_{j-1}+2} < \ldots \\
& \leq \lambda_{i_{k_{j}-1}} < \hat{y}_{k_{j}-1} < \hat{y}_{k_{j}-1}+\hat{a}_{k_{j}-1} < \lambda_{i_{k_{j}-1}+1} \leq \lambda_{i_{k_{j}}} < \hat{y}_{k_{j}} < \hat{y}_{k_{j}}+\hat{a}_{k_{j}} < \lambda_{i_{k_{j}+1}} \leq \mathfrak{b}_{j}.
\end{align*}
Thus $i_{1},\ldots,i_{k}\in \Lambda(I)$ and \eqref{lol33} holds for all $j=1,\ldots,k$. Moreover, since $\mathfrak{b}_{j} < \mathfrak{a}_{j+1}$ for $j=1,\ldots,p-1$, we have $i_{1}<i_{2}<\ldots<i_{k}$, and therefore $(y_{1},\ldots,y_{k})\in \Sigma_{k}(a_{1},\ldots,a_{k})$.
\end{proof}

%%%%%%%%%%%%%%%%%%%%%%%%%%%%%%%%%%%%%%%%%%%%% 
 
\paragraph{Acknowledgements.} 
The author is a Research Associate of the Fonds de la Recherche Scientifique - FNRS, and also acknowledges support from the European Research Council (ERC), Grant Agreement No. 101115687. %Views and opinions expressed are those of the author only and do not necessarily reflect those of the European Union or the European Research Council. Neither the European Union nor the granting authority can be held responsible for them.


\footnotesize
\begin{thebibliography}{99}
\bibitem{AMA2014} M. Amini, S.M.T.K. MirMostafaee and J. Ahmadi, Log-gamma-generated families of distributions, \textit{Statistics} \textbf{48} (2014), no. 4, 913--932.

\bibitem{AGZ2010} G.W. Anderson, A. Guionnet and O. Zeitouni, \textit{An introduction to random matrices}, Cambridge Stud. Adv. Math., 118, Cambridge University Press, Cambridge, 2010. xiv+492 pp.

\bibitem{BAB2013} G. Ben Arous and P. Bourgade, Extreme gaps between eigenvalues of random matrices, \textit{Ann. Probab.} \textbf{41} (2013), no. 4, 2648--2681.

\bibitem{B2022} P. Bourgade, Extreme gaps between eigenvalues of Wigner matrices, \textit{J. Eur. Math. Soc.} \textbf{24} (2022), no. 8, 2823--2873.

\bibitem{C2019} C. Charlier, Asymptotics of Hankel determinants with a one-cut regular potential and Fisher-Hartwig singularities, \textit{Int. Math. Res. Not. IMRN} \textbf{2019} (2019), no. 24, 7515--7576.

\bibitem{C2025} C. Charlier, Smallest gaps of the two-dimensional Coulomb gas, arXiv:2507.23502.

\bibitem{CG2021} C. Charlier and R. Gharakhloo, Asymptotics of Hankel determinants with a Laguerre-type or Jacobi-type potential and Fisher-Hartwig singularities, \textit{Adv. Math.} \textbf{383} (2021), Paper No. 107672, 69 pp.

\bibitem{CFWW2025} C. Charlier, B. Fahs, C. Webb and M.D. Wong, Asymptotics of Hankel determinants with a multi-cut regular potential and Fisher-Hartwig singularities, \textit{Mem. Amer. Math. Soc.} \textbf{310} (2025), no. 1567, v+138 pp.

\bibitem{CFLW2021} T. Claeys, B. Fahs, G. Lambert and C. Webb, How much can the eigenvalues of a random Hermitian matrix fluctuate?, \textit{Duke Math. J.} \textbf{170} (2021), no. 9, 2085--2235.

\bibitem{DL2025} D. Dai and C. Lu, On The Eigenvalue Rigidity of the Jacobi Unitary Ensemble, arXiv:2511.18967.

\bibitem{Deift} P. Deift, \textit{Orthogonal polynomials and random matrices: a Riemann-Hilbert approach},
Courant Lecture Notes in Mathematics, 3. New York University, Courant Institute of Math-
ematical Sciences, New York; American Mathematical Society, Providence, RI, 1999

\bibitem{DIKZ2007} P. Deift, A. Its, I. Krasovsky and X. Zhou, The Widom-Dyson constant for the gap probability in random matrix theory, \textit{J. Comput. Appl. Math.} \textbf{202} (2007), no. 1, 26--47.

\bibitem{DKMVZ} P. Deift, T. Kriecherbauer, K. T.-R. McLaughlin, S. Venakides and X. Zhou, Uniform asymptotics for polynomials orthogonal with respect to varying exponential weights and applications to universality questions in random matrix theory, \textit{Comm. Pure Appl. Math.} \textbf{52} (1999), no. 11, 1335--1425.

\bibitem{Dyson} F.J. Dyson, Statistical theory of energy levels of complex systems I, \textit{J. Math.Phys.} \textbf{3}, 140--156 (1962).

\bibitem{EYY} L. Erd\"{o}s, H.-T. Yau and J. Yin, Rigidity of eigenvalues of generalized Wigner matrices,
\textit{Adv. Math.} \textbf{229} (2012), no. 3, 1435--1515.

\bibitem{FGY2024} R. Feng, F. G\"{o}tze and D. Yao, Smallest gaps between zeros of stationary Gaussian processes, \textit{J. Funct. Anal.} \textbf{287} (2024), no. 4, Paper No. 110493, 14 pp. 

\bibitem{FLY2024} R. Feng, J. Li and D. Yao, Small gaps of GSE, arXiv:2409.03324.

\bibitem{FM2026} R. Feng and S. Muirhead, Poisson approximation of the largest gaps between zeros of a stationary Gaussian process, preprint.

\bibitem{FTW2019} R. Feng, G. Tian and D. Wei, Small gaps of GOE, \textit{Geom. Funct. Anal.} \textbf{29} (2019), no. 6, 1794--1827.

\bibitem{FW2021} R. Feng and D. Wei, Small gaps of circular $\beta$-ensemble, \textit{Ann. Probab.} \textbf{49} (2021), no. 2, 997--1032.

\bibitem{FW2018} R. Feng and D. Wei, Large gaps of CUE and GUE, \textit{Ann. Probab.} \textbf{53} (2025), no. 5, 1668--1702.

\bibitem{FY2023} R. Feng and G. Yao, Smallest distances between zeros of Gaussian analytic functions, DOI: 10.13140/RG.2.2.26674.62409. (to appear in \textit{Advances in Mathematics})

\bibitem{FG2016} A. Figalli and A. Guionnet, Universality in several-matrix models via approximate transport maps, \textit{Acta Math.} \textbf{217} (2016), no. 1, 81--176.

\bibitem{ForresterAiryBessel} P.J. Forrester, The spectrum edge of random matrix ensembles, \textit{Nuclear Phys. B} \textbf{402} (1993), no. 3, 709--728.

\bibitem{F book} P.J. Forrester, Log-gases and random matrices, Princeton University Press, Princeton, NJ,
2010.

\bibitem{LLM2020} B. Landon, P. Lopatto and J. Marcinek, Comparison theorem for some extremal eigenvalue statistics, \textit{Ann. Probab.} \textbf{48} (2020), no. 6, 2894--2919.

\bibitem{LM2024} P. Lopatto and M. Meeker, Smallest gaps between eigenvalues of real Gaussian matrices, \textit{Bernoulli} \textbf{32} (2026), no. 1.

\bibitem{LO2025} P. Lopatto and M. Otto, Maximum gap in complex Ginibre matrices, arXiv:2501.04611.

\bibitem{Mehta} M.L. Mehta, \textit{Random matrices}, Third edition, Pure Appl. Math. (Amst.), 142, Elsevier/Academic Press, Amsterdam, 2004.


%\bibitem{NIST} F.W.J. Olver, A.B. Olde Daalhuis, D.W. Lozier, B.I. Schneider, R.F. Boisvert, C.W. Clark, B.R. Miller and B.V. Saunders, NIST Digital Library of Mathematical Functions. http://dlmf.nist.gov/, Release 1.0.13 of 2016-09-16.

\bibitem{SaTo} E. B. Saff and V. Totik, {\em Logarithmic potentials with external fields},  Grundlehren der Mathematischen Wissenschaften, Springer-Verlag, Berlin, 1997.

\bibitem{SoshnikovSurvey} A. Soshnikov, Determinantal random point fields, \textit{ Russian Math. Surveys} \textbf{55} (2000), no. 5, 923--975.

\bibitem{Soshnikov} A. Soshnikov, Statistics of extreme spacing in determinantal random point processes, \textit{Mosc. Math. J.} \textbf{5} (2005), no. 3, 705--719.

\bibitem{SJ2012} D. Shi and Y. Jiang, Smallest gaps between eigenvalues of random matrices with complex Ginibre, Wishart and universal unitary ensembles, arXiv:1207.4240.

\bibitem{TW Airy} C.A. Tracy and H. Widom, Level-spacing distributions and the Airy kernel, \textit{Comm. Math. Phys.} \textbf{159} (1994), no. 1, 151--174.

\bibitem{TW Bessel} C.A. Tracy and H. Widom, Level spacing distributions and the Bessel kernel, \textit{Comm. Math. Phys.} \textbf{161} (1994), no. 2, 289--309.

\bibitem{Vinson2001} J. Vinson, Closest spacing of eigenvalues, Ph.D. thesis, Princeton University, 2001.

\bibitem{Zhang} A. Zhang, Quantitative gap universality for Wigner matrices, arXiv:2507.20442.

\end{thebibliography}
\end{document}